\documentclass[11pt]{article}

\usepackage[a4paper, margin=1in]{geometry}
\usepackage{setspace}
\usepackage{amsmath, amssymb, amsthm, bm}
\usepackage{mathrsfs}  

\usepackage{graphicx}
\usepackage{hyperref}  
\hypersetup{
	colorlinks=true, 
	linkcolor=blue, 
	citecolor=blue, 
	urlcolor=blue    
}
\usepackage{booktabs}
\newtheoremstyle{mythm}
{10pt}{10pt}
{\itshape}{}
{\bfseries}{.}{ }
{\thmname{#1}\thmnumber{ #2}\thmnote{ (#3)}}

\theoremstyle{mythm}
\newtheorem{theorem}{Theorem}
\newtheorem{lemma}{Lemma}

\theoremstyle{definition}

\newtheorem{remark}{Remark}

\usepackage{titling}
\pretitle{\begin{center}\LARGE\bfseries}
	\posttitle{\par\end{center}\vspace{0.5em}}
\preauthor{\begin{center}\normalsize}
	\postauthor{\par\end{center}\vspace{-1em}}

\title{A Non-compact Positivity-Preserving Scheme for Parabolic PDE via Conditional Expectation}

\author{
	Haoran Xu\thanks{School of Mathematical Sciences, Soochow University, Email: 20244007005@stu.suda.edu.cn} \\
	Jie Ren\thanks{Center for Financial Engineering, School of Mathematical Sciences, Soochow University, Email: 731563573@qq.com} \\
	\textbf{Corresponding Author:} Xingye Yue\thanks{Center for Financial Engineering, School of Mathematical Sciences, Soochow University, Email: xyyue@suda.edu.cn (Corresponding Author)}
}
\date{\today}

\begin{document}
	
	\maketitle

	\begin{abstract}
		We propose a novel non-compact, positivity-preserving scheme for linear non-divergence form parabolic equations. Based on the Feynman–Kac formula, the solution is expressed as a conditional expectation of an associated diffusion process. Instead of using compact Markov chain approximations, we employ a wide stencil scheme to approximate the conditional expectation, ensuring consistency and positivity preservation. This method is  effective for anisotropic diffusion with mixed derivatives, where classical schemes often fail,unless the covariance matrix is diagonal dominated.
		
		A key feature of our framework is its robust treatment of boundary conditions, which avoids
		the accuracy loss commonly encountered in BZ and semi-Lagrangian schemes. For Dirichlet boundaries, we introduce (i) a quad-tree non-uniform stopping time scheme with $O(\Delta t^{1/2})$ accuracy and (ii) a quad-tree uniform stopping time scheme with $O(\Delta t)$ accuracy. For Neumann boundaries, we use discrete specular reflection with $O(\Delta t^{1/2})$ convergence, while periodic boundaries are treated using modular wrapping, achieving $O(\Delta t)$ accuracy. All analyses are conducted under the practical scaling $\Delta t \sim h$.
		
		Except for the  uniform stopping time scheme, all schemes are explicit.The schemes are  unconditionally stable and positive preserving, thanks to the probabilistic structure. To make sure the consistency, non-compact stencil is involved, which leads to the large time step constrain $\Delta t \sim h$. Numerical experiments confirm the predicted $L^\infty$ convergence rates for all types of boundary conditions.
	\end{abstract}
	\section{Introduction}\label{sec1}
	The numerical simulation of physical, engineering, and financial systems
	often involves quantities that must remain non-negative, such as
	probability densities, mass concentrations, and option prices.
	While continuous models governing these systems typically preserve
	positivity, assuming non-negative initial and boundary data, maintaining
	this property after discretization presents a significant challenge for
	many numerical methods.
	Positivity is essential not only for numerical accuracy but also for
	fundamental properties such as monotonicity, the comparison principle,
	and the maximum principle.
	For linear problems, positivity and monotonicity are equivalent; however,
	for nonlinear systems, positivity is a necessary, but not sufficient,
	condition for monotonicity.
	In discrete schemes, monotonicity generally ensures stability in the
	maximum norm, which is critical for practical computations.
	For complex nonlinear problems, the theory of viscosity solutions
	provides the theoretical framework, with a key result guaranteeing that
	consistent and monotone schemes will converge to the viscosity solution
	of the continuous problem \cite{bib1}.
	
	In one spatial dimension, standard methods such as finite differences,
	finite elements, and finite volumes can often be tuned to preserve
	positivity through careful mesh ratio selection.
	The situation becomes considerably more difficult in multiple
	dimensions, particularly for problems with anisotropic diffusion
	involving mixed second-order derivative terms.
	It is well-established that for general two-dimensional linear
	anisotropic problems, a linear compact 9-point scheme cannot
	simultaneously achieve consistency and positivity preservation without a
	diagonally dominant diffusion matrix
	\cite{bib11,bib3,bib13,bib12}.
	This fundamental limitation has motivated the development of two primary
	research directions for constructing globally consistent and
	positivity-preserving schemes:
	\begin{enumerate}
		\item \textbf{Nonlinear compact schemes.}
		Introducing nonlinearity into the scheme allows the enforcement of
		positivity while maintaining a compact stencil.
		This approach has been successfully applied in finite difference
		\cite{bib13}, finite element \cite{bib10}, and finite volume methods
		\cite{bib14,bib17,bib18,bib7,bib20,bib21}.
		
		\item \textbf{Non-compact linear schemes.}
		By relaxing the compactness requirement, linear schemes can be
		constructed that are both consistent and positivity-preserving.
		Prominent examples include the BZ method \cite{bib2} and the
		semi-Lagrangian method \cite{bib4}.
		Their core idea is to decompose the diffusion operator into a sum of
		second-order directional derivatives along directions not aligned with
		the grid.
		While this eliminates mixed derivatives, it necessitates the use of
		wide stencils and requires positivity-preserving interpolation for
		off-grid evaluations.
		A major practical challenge lies in boundary treatment: the extended
		directional stencils frequently reach outside the computational domain,
		often forcing extrapolation that can impair accuracy and monotonicity.
		However, in certain financial applications, the specific structure of
		the coefficients and the solution combined with techniques such as
		local coordinate rotation can be exploited to design consistent and
		monotone discretizations that preserve precision at the boundaries, as
		demonstrated, for instance, in \cite{bib24}.
		In practice, such approaches lack general applicability.
	\end{enumerate}
	
	Xu et al.\ \cite{bib16} developed a compact, positivity-preserving linear
	scheme for a class of degenerate Fokker--Planck equations (with
	off-diagonally dominant covariance matrices at degenerate points) by
	exploiting the problem's special structure.
	However, this approach is limited to this specific case and does not work
	for general anisotropic problems without diagonal dominance.
	
	Probabilistic methods provide an alternative approach by using the
	Feynman--Kac formula to express the PDE solution as a conditional
	expectation.
	Classical tree methods in finance \cite{bib5} and Markov chain
	approximations for stochastic control \cite{bib9} belong to this
	category.
	However, these methods typically result in compact stencils equivalent
	to standard finite difference schemes, which face the same diagonal
	dominance restrictions for multi-dimensional anisotropic problems.
	While the explicit semi-Lagrangian method \cite{bib4} can also be viewed
	from a stochastic-process perspective, it struggles with effective
	boundary handling.
	In addition, other works have employed probabilistic representations for
	high-dimensional nonlinear PDEs \cite{bib6,bib8,bib22}.
	
	\medskip
	\noindent\textbf{Our contribution.}
	This paper presents a novel non-compact numerical framework for linear
	non-divergence form parabolic equations that preserves positivity and
	ensures consistency for general anisotropic problems with mixed
	derivatives.
	Based on the Feynman--Kac formula, the method combines a wide stencil
	approximation with a positivity-preserving interpolation scheme.
	Extending our previous one-dimensional work \cite{bib23}, we develop a
	complete two-dimensional framework, including algorithmic details and a
	rigorous convergence analysis.
	Unlike previous non-compact approaches—such as the BZ method \cite{bib2}
	and the semi-Lagrangian method \cite{bib4}—which often experience
	accuracy loss near boundaries, our method naturally incorporates
	tailored treatments for Dirichlet, Neumann, and periodic conditions
	directly within the expectation-based update.
	This leads to unconditionally $L^\infty$ stable schemes that handle
	regions near the boundary accurately without resorting to extrapolation.
	To ensure consistency, a non-compact stencil is involved, leading to a
	large time step constraint $\Delta t \sim h$.
	Theoretically, we establish convergence results with $L^\infty$ error
	estimates of order $O(\Delta t^{1/2})$ or $O(\Delta t)$.
	Under the scaling $\Delta t \sim h$, the uniform stopping time scheme
	achieves an $L^\infty$ error of $O(\Delta t)$, while all other schemes
	yield $O(\Delta t^{1/2})$.
	Moreover, with the exception of the uniform stopping time scheme, all
	schemes are fully explicit.
	Numerical experiments confirm the predicted convergence rates and
	demonstrate that our method accurately resolves solutions near Dirichlet
	boundaries.
	In contrast, the explicit LISL scheme (the semi-Lagrangian method with
	linear interpolation \cite{bib4}) exhibits accuracy degradation in the
	regions near the boundary.
	
	\medskip
	The paper is structured as follows.
	Section~\ref{sec2} presents the schemes, stability, and convergence
	analysis.
	Section~\ref{sem} compares it with the semi-Lagrangian method.
	Section~\ref{sec4} provides numerical experiments to validate the
	convergence rates and demonstrate the robustness of the method.
	Finally, Section~\ref{sec5} concludes the paper.
	\section{Numerical Scheme}
	\label{sec2}
	
	\subsection{First-Type (Dirichlet) Boundary Conditions}
	\label{subsec:first-type}
	
	We consider the backward parabolic problem
	\begin{equation}\label{eq:PDE-backward}
		\begin{cases}
			\partial_t f + \dfrac{1}{2}\,\mathrm{Tr}\big( A A^\top D^2 f \big) + B\cdot\nabla f - r(x,y,t)\,f = 0,
			& (x,y)\in\Omega,\; t\in[0,T),\\[2pt]
			f(x,y,T) = \varphi(x,y), & (x,y)\in\Omega,\\[1pt]
			f\big|_{\partial\Omega} = f_{\partial}(x,y,t), & t\in[0,T),
		\end{cases}
	\end{equation}
	on the rectangular spatial domain $\Omega=(x_0,x_{M_1})\times (y_0,y_{M_2})$.
	The diffusion tensor and convection vector are
	\[
	A A^\top =
	\begin{pmatrix}
		\sigma_1^2 & \rho\,\sigma_1\sigma_2 \\[1pt]
		\rho\,\sigma_1\sigma_2 & \sigma_2^2
	\end{pmatrix},\qquad
	B=(b_1,b_2)^\top,
	\]
	and the Hessian and gradient are
	\[
	D^2 f = 
	\begin{pmatrix} 
		\partial_{xx} f & \partial_{xy} f \\ 
		\partial_{yx} f & \partial_{yy} f 
	\end{pmatrix}, \quad 
	\nabla f = 
	\begin{pmatrix} 
		\partial_x f \\ 
		\partial_y f 
	\end{pmatrix}.
	\]
	
	We assume $\rho,\sigma_1,\sigma_2,r,b_1,b_2$ are continuous on $\Omega\times[0,T)$ with 
	$\rho\in[-1,1]$, $\sigma_1>0$, $\sigma_2>0$, and $r\ge0$.
	A convenient factorization is
	\[
	A =
	\begin{pmatrix}
		\sigma_1\cos\theta & \sigma_1\sin\theta\\[1pt]
		\sigma_2\sin\theta & \sigma_2\cos\theta
	\end{pmatrix}, \qquad
	\theta = \dfrac12 \arcsin \rho \in \Big[-\dfrac{\pi}{4}, \dfrac{\pi}{4}\Big].
	\]
	
	Introduce the uniform grid
	\[
	\begin{gathered}
		x_i = x_0 + i h_1, \quad i=0,\dots,M_1,\quad h_1 = \dfrac{x_{M_1}-x_0}{M_1},\\
		y_j = y_0 + j h_2, \quad j=0,\dots,M_2,\quad h_2 = \dfrac{y_{M_2}-y_0}{M_2},\\
		t_n = n\Delta t, \quad n=0,\dots,N,\qquad \Delta t = \dfrac{T}{N},
	\end{gathered}
	\]
	and denote the discrete approximation by $f_{h,i,j}^n \approx f(x_i,y_j,t_n)$.
	
	By the two-dimensional Feynman--Kac representation, for $t_n \le t \le t_{n+1}$,
	\begin{equation}\label{eq:Feynman-Kac}
		f(x_i,y_j,t_n)
		= \mathbb{E}\Bigg[ 
		\exp\Big(-\int_{t_n}^{\tau} r(X_s,s)\,\mathrm{d}s\Big) f(X_{\tau},Y_{\tau},\tau)
		\;\Big|\; X_{t_n}=x_i,\, Y_{t_n}=y_j
		\Bigg],
	\end{equation}
	where $(X_s,Y_s)$ solves the SDE system
	\begin{equation}\label{eq:SDE-system}
		\begin{aligned}
			\mathrm{d}X(s) &= b_1(X_s,Y_s,s)\,\mathrm{d}s 
			+ \sigma_1(X_s,Y_s,s)\cos\theta(X_s,Y_s,s)\,\mathrm{d}W_1(s)
			+ \sigma_1(X_s,Y_s,s)\sin\theta(X_s,Y_s,s)\,\mathrm{d}W_2(s),\\
			\mathrm{d}Y(s) &= b_2(X_s,Y_s,s)\,\mathrm{d}s 
			+ \sigma_2(X_s,Y_s,s)\sin\theta(X_s,Y_s,s)\,\mathrm{d}W_1(s)
			+ \sigma_2(X_s,Y_s,s)\cos\theta(X_s,Y_s,s)\,\mathrm{d}W_2(s),
		\end{aligned}
	\end{equation}
	with independent standard Brownian motions $W_1,W_2$, and the stopping time
	\begin{equation}\label{eq:tau-def}
		\tau = t_{n+1} \wedge \inf\{\,s\ge t_n:\ (X_s,Y_s)\notin\Omega\,\}.
	\end{equation}
	
	At node $(x_i,y_j,t_n)$, define
	\[
	\begin{aligned}
		b_{k,i,j}^n &= b_k(x_i,y_j,t_n),       & \sigma_{k,i,j}^n &= \sigma_k(x_i,y_j,t_n), & \theta_{i,j}^n &= \theta(x_i,y_j,t_n), \\
		\alpha_{i,j}^n &= \sin\theta_{i,j}^n + \cos\theta_{i,j}^n, & 
		\beta_{i,j}^n &= \sin\theta_{i,j}^n - \cos\theta_{i,j}^n, \\
		r_{i,j}^n &= r(x_i,y_j,t_n),          & \rho_{i,j}^n &= \rho(x_i,y_j,t_n).
	\end{aligned}
	\]
	
	Then we have 
	\[
	\alpha_{i,j}^n \in [0,\sqrt2],\quad 
	\beta_{i,j}^n \in [-\sqrt2,0],\quad
	(\alpha_{i,j}^n)^2+(\beta_{i,j}^n)^2=2,\quad
	(\alpha_{i,j}^n)^2-(\beta_{i,j}^n)^2=2\rho_{i,j}^n.
	\]
	
	To approximate \eqref{eq:Feynman-Kac}, integrate \eqref{eq:SDE-system} from $t_n$ to $t$ and apply a first-order approximation:
	\begin{equation}\label{eq:process-approx}
		\begin{aligned}
			X_t &\approx x_i + b_{1,i,j}^n (t-t_n)
			+ \sigma_{1,i,j}^n \cos\theta_{i,j}^n \sqrt{t-t_n}\,\xi_1
			+ \sigma_{1,i,j}^n \sin\theta_{i,j}^n \sqrt{t-t_n}\,\xi_2,\\
			Y_t &\approx y_j + b_{2,i,j}^n (t-t_n)
			+ \sigma_{2,i,j}^n \sin\theta_{i,j}^n \sqrt{t-t_n}\,\xi_1
			+ \sigma_{2,i,j}^n \cos\theta_{i,j}^n \sqrt{t-t_n}\,\xi_2,
		\end{aligned}
	\end{equation}
	where $\xi_1$ and $\xi_2$ are standard normal random variables.  
	We replace each $\xi$ by an independent Rademacher variable
	\[
	\xi^h =
	\begin{cases}
		-1,& \text{w.p. } 1/2,\\[1pt]
		+1,& \text{w.p. } 1/2,
	\end{cases}
	\]
	so that the discrete process admits four equiprobable branches at time $t$:
	\[
	(X_t^h,Y_t^h) =
	\begin{cases}
		\big(x_i + b_{1,i,j}^n (t-t_n) + \alpha_{i,j}^n\sigma_{1,i,j}^n \sqrt{t-t_n},\ 
		y_j + b_{2,i,j}^n (t-t_n) + \alpha_{i,j}^n\sigma_{2,i,j}^n \sqrt{t-t_n}\big),\\[1pt]
		\big(x_i + b_{1,i,j}^n (t-t_n) - \beta_{i,j}^n\sigma_{1,i,j}^n \sqrt{t-t_n},\ 
		y_j + b_{2,i,j}^n (t-t_n) + \beta_{i,j}^n\sigma_{2,i,j}^n \sqrt{t-t_n}\big),\\[1pt]
		\big(x_i + b_{1,i,j}^n (t-t_n) - \alpha_{i,j}^n\sigma_{1,i,j}^n \sqrt{t-t_n},\ 
		y_j + b_{2,i,j}^n (t-t_n) - \alpha_{i,j}^n\sigma_{2,i,j}^n \sqrt{t-t_n}\big),\\[1pt]
		\big(x_i + b_{1,i,j}^n (t-t_n) + \beta_{i,j}^n\sigma_{1,i,j}^n \sqrt{t-t_n},\ 
		y_j + b_{2,i,j}^n (t-t_n) - \beta_{i,j}^n\sigma_{2,i,j}^n \sqrt{t-t_n}\big).
	\end{cases}
	\]
	
	Define the discrete stopping time
	\[
	\tau^h = t_{n+1} \wedge \inf\{\,s\ge t_n:\ (X_s^h,Y_s^h)\notin\Omega\,\}.
	\]
	As there are four discrete branches, denote the possible realizations by $\tau_k$, $k=1,\dots,4$, and let $\hat{\tau}_k = \tau_k - t_n$.
	
	A natural approximation of \eqref{eq:Feynman-Kac} is
	\begin{equation}\label{eq:naive-approx}
		\hat f_{h,i,j}^n
		= \mathbb{E}^h \Bigg[\frac{f(X_{\tau^h}^{h}, Y_{\tau^h}^{h}, \tau^h)}{1 + r_{i,j}^n \hat\tau^h}\Bigg]
		= \sum_{k=1}^4 \frac{1}{4 (1 + r_{i,j}^n \hat\tau_k)} 
		f(X_{\tau_k}^{h,k}, Y_{\tau_k}^{h,k}, \tau_k).
	\end{equation}
	Equation \eqref{eq:naive-approx} may suffer boundary-induced loss of accuracy.
	
	\begin{remark}
		As an illustrative example, consider $b_1=b_2=0$, $\sigma_1=\sigma_2=\rho=1$, and $r=0$.  
		For a grid point $(x_k, y_{k'})$, the corresponding four branch points are
		\[
		\begin{aligned}
			(x_{\tau_1}^{h,1}, y_{\tau_1}^{h,1}) 
			&= (x_k + \sqrt{2\tau_1},\; y_{k'} + \sqrt{2\tau_1}),\\
			(x_{\Delta t}^{h,2}, y_{\Delta t}^{h,2}) 
			&= (x_k,\, y_{k'}),\\
			(x_{\Delta t}^{h,3}, y_{\Delta t}^{h,3}) 
			&= (x_k - \sqrt{2\Delta t},\; y_{k'} - \sqrt{2\Delta t}),\\
			(x_{\Delta t}^{h,4}, y_{\Delta t}^{h,4}) 
			&= (x_k,\, y_{k'}).
		\end{aligned}
		\]
	\end{remark}
	To analyze the consistency, we assume that the function $f$ is sufficiently smooth. Then the scheme yields
	\[
	\begin{aligned}
		\hat f_{h,k,k'}^n - f(x_k, y_{k'}, t_n)
		&= \tfrac14\,f(x_{\tau_1}^{h,1}, y_{\tau_1}^{h,1}, \tau_1)
		+ \tfrac14\,f(x_{\Delta t}^{h,2}, y_{\Delta t}^{h,2}, t_{n+1})  \\
		&\quad+ \tfrac14\,f(x_{\Delta t}^{h,3}, y_{\Delta t}^{h,3}, t_{n+1})
		+ \tfrac14\,f(x_{\Delta t}^{h,4}, y_{\Delta t}^{h,4}, t_{n+1})
		- f(x_k, y_{k'}, t_n).
	\end{aligned}
	\]
	
	Applying Taylor expansion at each evaluation point and substituting back, we obtain the local error estimate
	\[
	\bigl| \hat f_{h,k,k'} - f(x_k, y_{k'}, t_n) \bigr|
	= \mathcal{O}(\sqrt{\Delta t}-\sqrt{\tau_1}).
	\]
	
	If $(x_k, y_{k'})$ is very close to the boundary, $\tau_1$ may reduce to the scale $O(\Delta t^2)$. Consequently, the local truncation error remains of order $\mathcal{O}(\sqrt{\Delta t})$, while the global error becomes $\mathcal{O}(1/\sqrt{\Delta t})$. Thus, such a scheme inevitably suffers from accuracy degradation near the boundary.
	
	A key observation is that although each branch is selected with equal probability $1/4$, the probabilities of the payoffs corresponding to the stopping times of each branch are generally not identical. Intuitively, branches with shorter stopping times reach the boundary earlier and thus exert a stronger influence on the solution at the current grid point. Therefore, it is natural to assign higher probabilities to those with shorter stopping times. This allows the scheme to adaptively and accurately capture boundary effects.
	
	Based on this consideration, we rewrite the scheme \eqref{eq:naive-approx} as
	\begin{equation}\label{eq:weighted-scheme}
		\hat f_{h,i,j}^n
		= \sum_{k=1}^4 \frac{\omega_k}{\,1+r_{i,j}^n\hat\tau_k\,}\,
		f(X_{\tau_k}^{h,k}, Y_{\tau_k}^{h,k}, \tau_k),
	\end{equation}
	where the probabilities $\{\omega_k\}$ are defined by
	\begin{equation}\label{eq:omega}
		\begin{aligned}
			\omega_1 &= \frac{\sqrt{\hat\tau_2 \hat\tau_3 \hat\tau_4}}
			{(\sqrt{\hat\tau_1}+\sqrt{\hat\tau_3})(\sqrt{\hat\tau_1 \hat\tau_3}+\sqrt{\hat\tau_2 \hat\tau_4})},\\
			\omega_2 &= \frac{\sqrt{\hat\tau_1 \hat\tau_3 \hat\tau_4}}
			{(\sqrt{\hat\tau_2}+\sqrt{\hat\tau_4})(\sqrt{\hat\tau_1 \hat\tau_3}+\sqrt{\hat\tau_2 \hat\tau_4})},\\
			\omega_3 &= \frac{\sqrt{\hat\tau_1 \hat\tau_2 \hat\tau_4}}
			{(\sqrt{\hat\tau_1}+\sqrt{\hat\tau_3})(\sqrt{\hat\tau_1 \hat\tau_3}+\sqrt{\hat\tau_2 \hat\tau_4})},\\
			\omega_4 &= \frac{\sqrt{\hat\tau_1 \hat\tau_2 \hat\tau_3}}
			{(\sqrt{\hat\tau_2}+\sqrt{\hat\tau_4})(\sqrt{\hat\tau_1 \hat\tau_3}+\sqrt{\hat\tau_2 \hat\tau_4})}.
		\end{aligned}
	\end{equation}
	
	This choice of $\omega_k$ ensures the consistency of the scheme, as will be seen in the proof of Lemma~\ref{a}. A direct algebraic check shows that $\omega_k \ge 0$ and $\sum_{k=1}^4 \omega_k = 1$. Moreover, we observe that
	\[
	\begin{aligned}
		&\text{If } \tau_1 = \tau_2 = \tau_3 = \tau_4 (= \Delta t),
		\text{ then } \omega_1 = \omega_2 = \omega_3 = \omega_4 = \tfrac14,\\
		&\text{If } \tau_1 = \Delta t^2, \; \tau_2 = \tau_3 = \tau_4 = \Delta t, \text{ then}\\
		&\qquad \omega_1 = \frac{1}{(\!1+\sqrt{\Delta t}\!)^2},\quad
		\omega_2 = \frac{\sqrt{\Delta t}}{2(\!1+\sqrt{\Delta t}\!)},\\
		&\qquad \omega_3 = \frac{\sqrt{\Delta t}}{(\!1+\sqrt{\Delta t}\!)^2},\quad
		\omega_4 = \frac{\sqrt{\Delta t}}{2(\!1+\sqrt{\Delta t}\!)}.
	\end{aligned}
	\]
	It can be seen that branches with smaller stopping times receive larger probabilities, while branches with equal stopping times are assigned equal probabilities, which is the desired property.
	
	\begin{lemma}\label{a}
		Assume that the exact solution $f$ is sufficiently smooth. Let $\hat f_{h,i,j}^n$ be the numerical approximation obtained from the weighted scheme \eqref{eq:weighted-scheme}. Then there exists a constant $C>0$ independent of $\Delta t$ such that
		\[
		\bigl|\hat f_{h,i,j}^n - f(x_i, y_j, t_n)\bigr| \;\le\; C\,\Delta t^{3/2}.
		\]
	\end{lemma}
	
	\begin{proof}
		We begin by defining the quadrature error term $I$ as
		\[
		\begin{aligned}
			I &= \hat f_{h,i,j}^n - f(x_i, y_j, t_n) \\
			&= \sum_{k=1}^4 \frac{\omega_k}{\,1 + r_{i,j}^n\hat{\tau}_k\,}\,
			f(X_{\tau_k}^{h,k}, Y_{\tau_k}^{h,k}, \tau_k)
			- f(x_i, y_j, t_n)  \\
			&= \Bigl[ A_1\,\partial_t f + A_2\,\partial_{xx}^2 f + A_3\,\partial_{xy}^2 f
			+ A_4\,\partial_{yy}^2 f + A_5\,\partial_x f + A_6\,\partial_y f + A_7\,f \Bigr]_{(x_i,y_j)}^{t_n}
			+ O(\Delta t^{3/2}),
		\end{aligned}
		\]
		where the coefficients $A_1,\dots,A_7$ are explicit algebraic combinations of the probabilities $\{\omega_k\}$, the time increments $\{\hat\tau_k\}$, and the local coefficients:
		\[
		\begin{aligned}
			A_1 &= \omega_1\,\hat{\tau}_1 + \omega_2\,\hat{\tau}_2
			+ \omega_3\,\hat{\tau}_3 + \omega_4\,\hat{\tau}_4,\\
			A_2 &= \tfrac{(\sigma_{1,i,j}^n)^2}{2}\,A_1 + \tfrac{\rho_{i,j}^n(\sigma_{1,i,j}^n)^2}{2}\,
			(\omega_1\,\hat{\tau}_1 - \omega_2\,\hat{\tau}_2 + \omega_3\,\hat{\tau}_3 - \omega_4\,\hat{\tau}_4),\\
			A_3 &= \rho_{i,j}^n\sigma_{1,i,j}^n\sigma_{2,i,j}^n\,A_1
			+ \sigma_{1,i,j}^n\sigma_{2,i,j}^n\,
			(\omega_1\,\hat{\tau}_1 - \omega_2\,\hat{\tau}_2 + \omega_3\,\hat{\tau}_3 - \omega_4\,\hat{\tau}_4),\\
			A_4 &= \tfrac{(\sigma_{2,i,j}^n)^2}{2}\,A_1 + \tfrac{\rho_{i,j}^n(\sigma_{2,i,j}^n)^2}{2}\,
			(\omega_1\,\hat{\tau}_1 - \omega_2\,\hat{\tau}_2 + \omega_3\,\hat{\tau}_3 - \omega_4\,\hat{\tau}_4),\\
			A_5 &= b_{1,i,j}^n\,A_1 + \omega_1\,\alpha_{i,j}^n\,\sigma_{1,i,j}^n\,\sqrt{\hat\tau_1}
			- \omega_2\,\beta_{i,j}^n\,\sigma_{1,i,j}^n\,\sqrt{\hat\tau_2}\\
			&\quad - \omega_3\,\alpha_{i,j}^n\,\sigma_{1,i,j}^n\,\sqrt{\hat\tau_3}
			+ \omega_4\,\beta_{i,j}^n\,\sigma_{1,i,j}^n\,\sqrt{\hat\tau_4},\\
			A_6 &= b_{2,i,j}^n\,A_1 + \omega_1\,\alpha_{i,j}^n\,\sigma_{2,i,j}^n\,\sqrt{\hat\tau_1}
			+ \omega_2\,\beta_{i,j}^n\,\sigma_{2,i,j}^n\,\sqrt{\hat\tau_2}\\
			&\quad - \omega_3\,\alpha_{i,j}^n\,\sigma_{2,i,j}^n\,\sqrt{\hat\tau_3}
			- \omega_4\,\beta_{i,j}^n\,\sigma_{2,i,j}^n\,\sqrt{\hat\tau_4},\\
			A_7 &= -\,r_{i,j}^n\,A_1.
		\end{aligned}
		\]
		
		From \eqref{eq:omega}, it follows that the probabilities $\{\omega_k\}$ satisfy the moment-matching conditions:
		\[
		\begin{aligned}
			\omega_1 + \omega_2 + \omega_3 + \omega_4 &= 1,\\
			\omega_1\,\hat{\tau}_1 - \omega_2\,\hat{\tau}_2
			+ \omega_3\,\hat{\tau}_3 - \omega_4\,\hat{\tau}_4 &= 0,\\
			\omega_1\,\alpha_{i,j}^n\,\sigma_{1,i,j}^n\,\sqrt{\hat\tau_1}
			- \omega_2\,\beta_{i,j}^n\,\sigma_{1,i,j}^n\,\sqrt{\hat\tau_2}
			- \omega_3\,\alpha_{i,j}^n\,\sigma_{1,i,j}^n\,\sqrt{\hat\tau_3}
			+ \omega_4\,\beta_{i,j}^n\,\sigma_{1,i,j}^n\,\sqrt{\hat\tau_4} &= 0,\\
			\omega_1\,\alpha_{i,j}^n\,\sigma_{2,i,j}^n\,\sqrt{\hat\tau_1}
			+ \omega_2\,\beta_{i,j}^n\,\sigma_{2,i,j}^n\,\sqrt{\hat\tau_2}
			- \omega_3\,\alpha_{i,j}^n\,\sigma_{2,i,j}^n\,\sqrt{\hat\tau_3}
			- \omega_4\,\beta_{i,j}^n\,\sigma_{2,i,j}^n\,\sqrt{\hat\tau_4} &= 0.
		\end{aligned}
		\]
		
		Substituting these identities into the expressions for $A_2$--$A_7$, we conclude that
		\[
		\begin{aligned}
			A_2 &= \tfrac{(\sigma_{1,i,j}^n)^2}{2}\,A_1,\qquad 
			A_3 = \rho_{i,j}^n\,\sigma_{1,i,j}^n\,\sigma_{2,i,j}^n\,A_1,\qquad
			A_4 = \tfrac{(\sigma_{2,i,j}^n)^2}{2}\,A_1,\\
			A_5 &= b_{1,i,j}^n\,A_1,\qquad 
			A_6 = b_{2,i,j}^n\,A_1,\qquad 
			A_7 = -\,r_{i,j}^n\,A_1.
		\end{aligned}
		\]
		
		Therefore, the error term simplifies to
		\[
		I = A_1\Bigl[\partial_t f + \tfrac{\sigma_1^2}{2}\,\partial_{xx}^2 f
		+ \rho\,\sigma_1\,\sigma_2\,\partial_{xy}^2 f 
		+ \tfrac{\sigma_2^2}{2}\,\partial_{yy}^2 f
		+ b_1\,\partial_x f + b_2\,\partial_y f - r\,f\Bigr]_{(x_i,y_j)}^{t_n}
		+ O(\Delta t^{3/2}).
		\]
		
		Since $f$ satisfies the original PDE \eqref{eq:PDE-backward}, the bracketed expression vanishes. Hence
		\[
		I = O(\Delta t^{3/2}),
		\]
		which completes the proof.
	\end{proof}
	In practice the scheme may require values at off-grid points at time level $t_{n+1}$; these are obtained by positivity-preserving interpolation from grid values. We introduce the following operators (used consistently below):
	
	\begin{itemize}
		\item $\mathcal{L}^{\mathrm{sp}}$: spatial bilinear interpolation (uses four spatial neighbors at a single time level);
		\item $\mathcal{L}^{\mathrm{st}}$: space--time trilinear interpolation (uses eight neighbors: four in space at $t_n$ and four in space at $t_{n+1}$).
	\end{itemize}
	
	Both operators are convex combinations of grid values when the evaluation point lies in the interior of a cell; near domain boundaries the interpolation stencils are adapted by replacing missing grid values with boundary data $f_{\partial}$.
	
	The practical update then reads
	\begin{equation}\label{eq:practical}
		f_{h,i,j}^{n}
		= \sum_{k=1}^4 \frac{\omega_k}{1+r_{i,j}^n\hat\tau_k}\;U(X_{\tau_k}^{h,k},Y_{\tau_k}^{h,k},\tau_k),
	\end{equation}
	with
	\[
	U(X_{\tau_k}^{h,k},Y_{\tau_k}^{h,k},\tau_k)=
	\begin{cases}
		\mathcal{L}^{\mathrm{sp}}f_h(X_{t_{n+1}}^{h,k},Y_{t_{n+1}}^{h,k},t_{n+1}), & (X_{t_{n+1}}^{h,k},Y_{t_{n+1}}^{h,k})\in\Omega,\\[2mm]
		f_{\partial}(X_{\tau_k}^{h,k},Y_{\tau_k}^{h,k},\tau_k), & (X_{\tau_k}^{h,k},Y_{\tau_k}^{h,k})\in\partial\Omega.
	\end{cases}
	\]
	
	Since both the linear interpolation operator and the mathematical expectation operator are positivity-preserving and linear, the resulting scheme is inherently linear and positivity-preserving. Moreover, as both linear interpolation and the computation of discrete mathematical expectations are based on nonnegative linear combinations whose coefficients sum to 1, the scheme is unconditionally stable in the $L^\infty$ norm.
	
	A key issue for the scheme is its consistency. By the Lax Equivalence Theorem, stability alone is insufficient; convergence is guaranteed only when consistency also holds. Unlike classical explicit finite difference schemes, which require CFL-type mesh-ratio constraints to maintain stability, our explicit construction is unconditionally stable. Establishing its consistency therefore becomes the main concern.
	
	When attempting to realize the scheme in a compact manner—so that each branch displacement remains within a single cell—one must enforce CFL-like restrictions of the form
	\[
	|\alpha\sigma_i\sqrt{\Delta t}| \le h_i,\qquad
	|\beta\sigma_i\sqrt{\Delta t}| \le h_i,\qquad i=1,2.
	\]
	Under these constraints, substituting the exact solution into the update and applying Lemma~\ref{a} yields a local truncation error of
	\[
	R_{\mathrm{loc}} = O(h_1^2+h_2^2) + O(\Delta t^{3/2}),
	\]
	and hence a global error of
	\[
	R_{\mathrm{glob}}
	= O\Big(\Delta t^{1/2} + \frac{h_1^2+h_2^2}{\Delta t}\Big).
	\]
	
	This expression makes the difficulty explicit: since bilinear interpolation is only second-order accurate, compact realizations generally cannot simultaneously satisfy the compatibility requirement and preserve positivity. In particular, the interpolation error vanishes only when the branch endpoints lie exactly on grid nodes. In other cases, the branch endpoints fall between grid points, producing a nonzero interpolation error that renders the scheme inconsistent.
	
	To avoid this incompatibility, we abandon compactness and adopt a non-compact scaling in which the time step is chosen on the order of the spatial mesh size,
	\[
	\Delta t \sim h_1 \sim h_2.
	\]
	In particular, taking
	\[
	\Delta t = h_1 = h_2
	\]
	ensures that the update formula remains positivity-preserving, linear, and fully compatible without imposing any CFL-type restriction. 
	
	\begin{remark}
		The global consistency error of the scheme is of the order $O\left(\Delta t^{1/2}+\frac{h^2}{\Delta t}\right)$. If we adopt the scaling relation $\Delta t \sim h^{4/3}$, the theoretical consistency error can be improved to $O(h^{2/3})$, which is marginally superior to the error order $O(\Delta t^{1/2}) = O(h^{1/2})$ under the scaling $\Delta t \sim h$. However, this scaling would result in a smaller time step, thereby compromising computational efficiency. Moreover, numerical experiments presented in Section \ref{sec4} demonstrate that the scheme with $\Delta t \sim h$ achieves nearly first-order convergence in the $L^\infty$ norm. Therefore, we select the scaling $\Delta t \sim h$ to balance accuracy performance and computational efficiency.
	\end{remark}
	
	\medskip
	\noindent
	\textbf{Algorithm 1(non-uniform stopping-time scheme \eqref{eq:practical}):}
	\begin{enumerate}
		\item \textbf{Initialization:} Set $f_{h,i,j}^N = \varphi(x_i, y_j)$ for all grid nodes.
		\item For $n = N-1, \dots, 0$:
		\begin{enumerate}
			\item For all interior grid nodes $(i,j)$:
			\begin{enumerate}
				\item Compute the four candidate branch endpoints and their stopping times $\{\tau_k\}_{k=1}^4$.
				\item Compute the probabilities $\{\omega_k\}_{k=1}^4$ using \eqref{eq:omega}.
				\item For each $k=1,\dots,4$:
				\begin{enumerate}
					\item If $\tau_k = t_{n+1}$ and $(X_{t_{n+1}}^{h,k}, Y_{t_{n+1}}^{h,k}) \in \Omega$, evaluate $U(X_{t_{n+1}}^{h,k}, Y_{t_{n+1}}^{h,k}, t_{n+1})$ via $\mathcal{L}^{\mathrm{sp}} f_h$.
					\item Else, evaluate $U(X_{\tau_k}^{h,k}, Y_{\tau_k}^{h,k}, \tau_k)$ as $f_\partial$.
				\end{enumerate}
				\item Update $f_{h,i,j}^n = \sum_{k=1}^4 \dfrac{\omega_k}{1 + r_{i,j}^n \hat{\tau}_k} \cdot U(X_{\tau_k}^{h,k}, Y_{\tau_k}^{h,k}, \tau_k)$.
			\end{enumerate}
		\end{enumerate}
	\end{enumerate}
	
	\begin{figure}[htbp]
		\centering
		\includegraphics[width=0.8\linewidth]{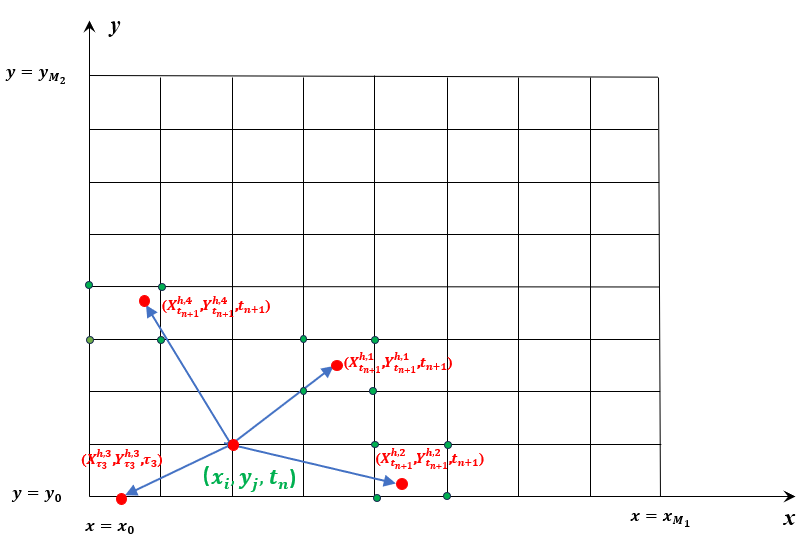}
		\caption{Illustration of the above algorithm, where $\tau_1 = \tau_2 = \tau_4 = t_{n+1}$ and $\tau_3 < t_{n+1}$.}
		\label{fig1}
	\end{figure}
	
	\begin{theorem}\label{thm:compact-conv}
		Let $f$ be sufficiently smooth and consider the mesh scaling $\Delta t = h_1 = h_2$. 
		Define the pointwise error at grid node $(i,j)$ and time level $n$ by
		\[
		e_{i,j}^n = f_{h,i,j}^n - f(x_i, y_j, t_n),
		\]
		and the corresponding maximum norm as
		\[
		\|e^n\|_\infty = \max_{\substack{0 \le i \le M_1 \\ 0 \le j \le M_2}} |e_{i,j}^n|.
		\]
		Then the numerical solution produced by scheme~\eqref{eq:practical} satisfies
		\[
		\|e^n\|_\infty\leq O( \Delta t^{1/2}).
		\]
	\end{theorem}
	
	\begin{proof}
		From the definition of the numerical update in scheme~\eqref{eq:practical}, the local error is defined as
		\begin{equation*}
			\begin{aligned}
				e_{i,j}^n 
				&= \sum_{k=1}^4 \frac{\omega_k}{1+r_{i,j}^n \hat{\tau}_k}
				U\!\left(X_{\tau_k}^{h,k}, Y_{\tau_k}^{h,k}, \tau_k\right)
				- f(x_i, y_j, t_n) \\
				&= I_1 + I_2 + I_3,
			\end{aligned}
		\end{equation*}
		where
		\begin{equation*}
			\begin{aligned}
				I_1 &= \sum_{k=1}^4 
				\frac{\omega_k}{1+r_{i,j}^n \hat{\tau}_k}
				f\!\left(X_{\tau_k}^{h,k}, Y_{\tau_k}^{h,k}, \tau_k\right)
				- f(x_i, y_j, t_n), \\[0.3em]
				I_2 &= \sum_{(X_{t_{n+1}}^{h,k},\, Y_{t_{n+1}}^{h,k}) \in \Omega}
				\frac{\omega_k}{1+r_{i,j}^n \Delta t}\,
				\mathcal{L}^{sp}(f_h - f)\!\left(X_{t_{n+1}}^{h,k}, Y_{t_{n+1}}^{h,k}, t_{n+1}\right), \\[0.3em]
				I_3 &= \sum_{(X_{t_{n+1}}^{h,k},\, Y_{t_{n+1}}^{h,k}) \in \Omega}
				\frac{\omega_k}{1+r_{i,j}^n \Delta t}\,
				(\mathcal{L}^{sp}f - f)\!\left(X_{t_{n+1}}^{h,k}, Y_{t_{n+1}}^{h,k}, t_{n+1}\right).
			\end{aligned}
		\end{equation*}
		
		We estimate each component:
		
		1. $I_1$ (Local consistency error): By Lemma \ref{a}, $|I_1| = O(\Delta t^{3/2})$.
		
		2. $I_2$ (Propagated numerical error): Stability of $\mathcal{L}^{sp}$ gives $|I_2| \le \|e^{n+1}\|_\infty$.
		
		3. $I_3$ (Interpolation error of exact solution): Bilinear interpolation gives $|I_3| = O(h_1^2+h_2^2) = O(\Delta t^2)$ under $\Delta t = h_1 = h_2$.
		
		Combining, we get
		\[
		|e_{i,j}^n| \le O(\Delta t^{3/2}) + \|e^{n+1}\|_\infty.
		\]
		Iterating backward from $n=N$, we obtain
		\[
		\|e^k\|_\infty \le (N-k)O(\Delta t^{3/2}) + \|e^N\|_\infty = O(\Delta t^{1/2}).
		\]
		
		This completes the proof.
	\end{proof}
	The global convergence rate of $O(\Delta t^{1/2})$ is constrained by the asymmetric truncation term $I_1$, which arises when different branches terminate at different times. To recover cancellation of these low-order error terms, we introduce a uniform stopping rule: when any branch encounters the boundary, all four branches are stopped simultaneously at the earliest stopping time.
	
	Define the uniform stopping time and its increment as
	\[
	\tau = \min_{1 \le k \le 4} \tau_k, \qquad \hat{\tau} = \tau - t_n.
	\]
	The corresponding branch endpoints are obtained by propagating the discrete process for each branch up to this common time $\tau$. The uniform update is then given by
	\begin{equation}\label{eq:symmetric}
		f_{h,i,j}^{n} = \frac{1}{4(1 + r_{i,j}^n \hat{\tau})} \sum_{k=1}^{4} U\left(X_{\tau}^{h,k}, Y_{\tau}^{h,k}, \tau\right),
	\end{equation}
	where the evaluation of $U(\cdot)$ depends on the endpoint location and time:
	\begin{itemize}
		\item For $\tau=t_{n+1}$, $U$ is evaluated using spatial interpolation $\mathcal{L}^{\mathrm{sp}}f_h$ at time $t_{n+1}$.
		\item For $\tau < t_{n+1}$, $U$ is evaluated via space-time trilinear interpolation $\mathcal{L}^{\mathrm{st}}$.
		\item For $(X_{\tau}^{h,k}, Y_{\tau}^{h,k}) \in \partial\Omega$, the boundary condition $f_{\partial}$ is applied.
	\end{itemize}
	
	\medskip
	\noindent
	\textbf{Algorithm 2 (uniform stopping time scheme \eqref{eq:symmetric}):}
	\begin{enumerate}
		\item \textbf{Initialization:} Set $f_{h,i,j}^N = \varphi(x_i, y_j)$.
		\item For $n = N-1, \dots, 0$:
		\begin{enumerate}
			\item For all interior nodes $(i,j)$:
			\begin{enumerate}
				\item Compute the four candidate endpoints and their stopping times $\{\tau_k\}$.
				\item Determine the uniform stopping time $\tau = \min_k \tau_k$ and update all $\tau_k \leftarrow \tau$.
				\item If $\tau = t_{n+1}$:
				\begin{enumerate}
					\item For each branch endpoint:
					\begin{itemize}
						\item If it lies on $\partial\Omega$: Use boundary value $f_{\partial}$.
						\item Otherwise: Evaluate $U$ via spatial interpolation $\mathcal{L}^{\mathrm{sp}}$ at $t_{n+1}$.
					\end{itemize}
					\item Update $f_{h,i,j}^n$ explicitly using \eqref{eq:symmetric}.
				\end{enumerate}
				\item Else ($\tau < t_{n+1}$):
				\begin{enumerate}
					\item For each branch endpoint:
					\begin{itemize}
						\item If it lies on $\partial\Omega$: Use boundary value $f_{\partial}$.
						\item Otherwise: Evaluate $U$ via space-time trilinear interpolation $\mathcal{L}^{\mathrm{st}}$.
					\end{itemize}
					\item Assemble the sparse linear system $T\mathbf{F} = \mathbf{b}$ for the unknowns coupled through $\mathcal{L}^{\mathrm{st}}$.
				\end{enumerate}
				\item Solve $T\mathbf{F} = \mathbf{b}$ .
			\end{enumerate}
		\end{enumerate}
	\end{enumerate}
	
	\begin{figure}[htbp]
		\centering
		\includegraphics[width=0.9\linewidth]{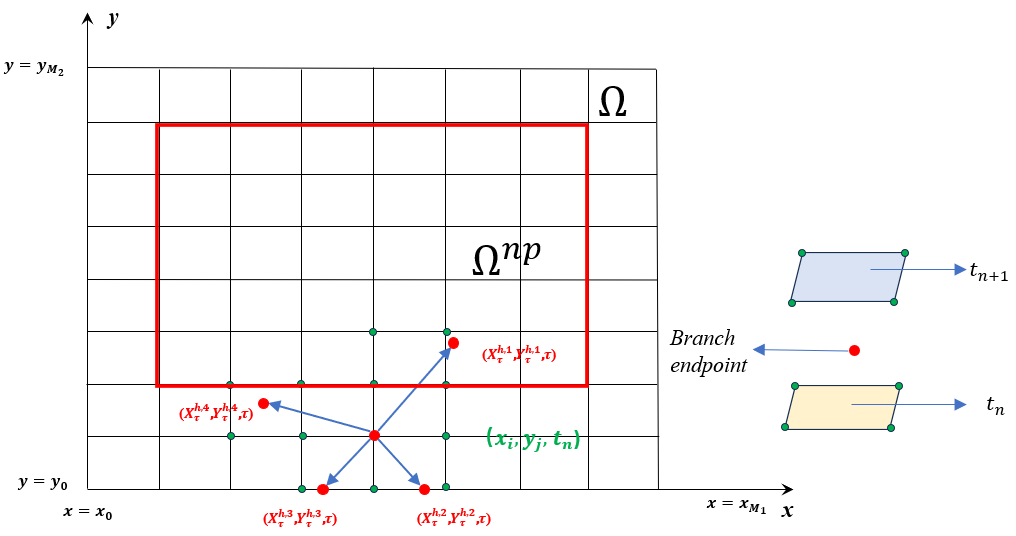}
		\caption{Illustration of the uniform stopping time algorithm where $\tau < t_{n+1}$. Here, $\Omega^{np}$ denotes the set of grid points reached from $t_n$ with a stopping time of $t_{n+1}$.}
		\label{fig2}
	\end{figure}
	
	Positivity is straightforward for nodes within the region $\Omega^{np}$. For nodes in $\Omega \setminus \Omega^{np}$, the implicit coupling leads to a sparse linear system for the corresponding unknowns, characterized by the matrix $T$. It can be readily verified that $T$ is an M-matrix, thereby ensuring the positivity-preserving property of the uniform stopping-time scheme.
	
	\begin{theorem}\label{thm:symmetric-conv}
		Assume that $f$ is sufficiently smooth and $\Delta t = h_1 = h_2$. 
		Then the uniform stopping time scheme~\eqref{eq:symmetric} satisfies the improved estimate
		\[
		\|e^n\|_\infty \le O(\Delta t).
		\]
	\end{theorem}
	
	\begin{proof}
		Analogously to the proof of Theorem~\ref{thm:compact-conv}, the grid-point error satisfies
		\[
		\begin{aligned}
			e_{i,j}^n 
			&= \sum_{k=1}^4 \frac{\omega_k}{1+r_{i,j}^n \hat{\tau}_k}
			U\!\left(X_{\tau_k}^{h,k}, Y_{\tau_k}^{h,k}, \tau_k\right)
			- f(x_i, y_j, t_n) \\
			&= I_1 + I_2 + I_3 ,
		\end{aligned}
		\]
		where $\alpha = \dfrac{t_{n+1}-\tau}{\Delta t}\in[0,1]$ and
		\[
		\begin{aligned}
			I_1 &= \sum_{k=1}^4 
			\frac{1}{4(1+r_{i,j}^n \hat{\tau})}
			f\!\left(X_{\tau}^{h,k}, Y_{\tau}^{h,k}, \tau\right)
			- f(x_i, y_j, t_n), \\[0.4em]
			I_2 &= \alpha\sum_{(X_{\tau}^{h,k}, Y_{\tau}^{h,k}) \in \Omega}
			\frac{1}{4(1+r_{i,j}^n \hat{\tau})}
			\mathcal{L}^{sp}(f_h - f)\left(X_{\tau}^{h,k}, Y_{\tau}^{h,k}, t_n\right) \\
			&\qquad + (1-\alpha)\sum_{(X_{\tau}^{h,k}, Y_{\tau}^{h,k}) \in \Omega}
			\frac{1}{4(1+r_{i,j}^n \hat{\tau})}
			\mathcal{L}^{sp}(f_h - f)\left(X_{\tau}^{h,k}, Y_{\tau}^{h,k}, t_{n+1}\right), \\[0.4em]
			I_3 &= 
			\sum_{(X_{\tau}^{h,k}, Y_{\tau}^{h,k}) \in \Omega} \frac{1}{4(1+r_{i,j}^n \hat{\tau})}
			(\mathcal{L}^{st} f - f)\left(X_{\tau}^{h,k}, Y_{\tau}^{h,k}, \tau\right).
		\end{aligned}
		\]
		
		We estimate each term sequentially:
		
		1. $I_1$: Analogously to Lemma~\ref{a}, $|I_1| = O(\Delta t^2)$.
		
		2. $I_3$: For smooth $f$, space-time trilinear interpolation gives
		\[
		|\mathcal{L}^{st} f - f| = O(h_1^2 + h_2^2 + \Delta t^2) = O(\Delta t^2) \quad \text{for } \Delta t = h_1 = h_2.
		\]
		
		3. $I_2$: Propagation of discrete error via the stable operator $\mathcal{L}^{st}$ gives
		\begin{itemize}
			\item If $\tau = t_{n+1}$, $|I_2| \le \|e^{n+1}\|_\infty$.
			\item If $\tau < t_{n+1}$, at least one branch intersects the boundary:
			\[
			|I_2| \le \frac{3}{4}\alpha \|e^n\|_\infty + \frac{3}{4}(1-\alpha)\|e^{n+1}\|_\infty.
			\]
		\end{itemize}
		
		Combining these, we get two alternative bounds:
		\[
		\| e^n \|_\infty \le \|e^{n+1}\|_\infty + O(\Delta t^2),
		\]
		or
		\[
		\| e^n \|_\infty
		\le 
		\frac{\frac{3}{4}(1-\alpha)}{1 - \frac{3}{4}\alpha}\|e^{n+1}\|_\infty
		+\frac{1}{1 - \frac{3}{4}\alpha} O(\Delta t^2).
		\]
		
		Using $0 \le \frac{\frac{3}{4}(1-\alpha)}{1 - \frac{3}{4}\alpha} \le 1$ and $1 \le \frac{1}{1 - \frac{3}{4}\alpha} \le 4$, both bounds give
		\[
		\|e^n\|_\infty \le \|e^{n+1}\|_\infty + O(\Delta t^2).
		\]
		
		Backward iteration from $n = N$ finally yields
		\[
		\|e^n\|_\infty \le O(\Delta t),
		\]
		completing the proof.
	\end{proof}
	\subsection{Reflective Boundary Conditions}
	
	The previous section addressed Dirichlet boundary conditions. We now consider homogeneous Neumann (reflective) boundary conditions, governed by the system:
	
	\begin{equation}  \label{eq:PDE-reflective}
		\begin{cases}
			\partial_t f + \tfrac{1}{2}\mathrm{Tr}\!\big(AA^\top D^2 f\big) + B\cdot\nabla f - r(x,y,t)\,f = 0,
			& (x,y)\in\Omega,\; t\in[0,T), \\
			f(\cdot,\cdot,T)=\varphi, & (x,y)\in\Omega, \\
			\dfrac{\partial f}{\partial n} = 0, & (x,y)\in \partial\Omega,t\in[0,T)
		\end{cases}
	\end{equation}
	with the same domain,notation and coefficient assumptions as before and $\frac{\partial f}{\partial n}$ is normal derivative on the boundary.
	
	In stochastic interpretation, homogeneous Neumann conditions correspond to specular reflection: branches that exit the domain \(\Omega\) are reflected back into the interior. Numerically, this is implemented by mirroring each proposed branch position that falls outside \(\Omega\) across the corresponding boundary.
	
	For interior nodes \(1\le i\le M_1-1,\;1\le j\le M_2-1\), the numerical update takes the form:
	\begin{equation} \label{eq:interior-update}
		f_{h,i,j}^{n} = \frac{1}{4(1+r_{i,j}^n\Delta t)} \sum_{k=1}^{4}
		\mathcal{L}^{\mathrm{sp}} f_h\big(X_k^h,Y_k^h,t_{n+1}\big),
	\end{equation}
	where \((X_k^h,Y_k^h)\) are the final branch positions obtained by applying the rules \eqref{eq:x-reflection},\eqref{eq:y-reflection} to the proposed branch points.
	
	The four proposed branch points are defined as:
	\begin{equation}
		\label{pr}
		\begin{aligned}
			(x_1^*, y_1^*) &= \left( x_i + b_{1,i,j}^n \Delta t + \alpha_{i,j}^n \sigma_{1,i,j}^n \sqrt{\Delta t},\; y_j + b_{2,i,j}^n \Delta t + \alpha_{i,j}^n \sigma_{2,i,j}^n \sqrt{\Delta t} \right), \\
			(x_2^*, y_2^*) &= \left( x_i + b_{1,i,j}^n \Delta t - \beta_{i,j}^n \sigma_{1,i,j}^n \sqrt{\Delta t},\; y_j + b_{2,i,j}^n \Delta t + \beta_{i,j}^n \sigma_{2,i,j}^n \sqrt{\Delta t} \right), \\
			(x_3^*, y_3^*) &= \left( x_i + b_{1,i,j}^n \Delta t - \alpha_{i,j}^n \sigma_{1,i,j}^n \sqrt{\Delta t},\; y_j + b_{2,i,j}^n \Delta t - \alpha_{i,j}^n \sigma_{2,i,j}^n \sqrt{\Delta t} \right), \\
			(x_4^*, y_4^*) &= \left( x_i + b_{1,i,j}^n \Delta t + \beta_{i,j}^n \sigma_{1,i,j}^n \sqrt{\Delta t},\; y_j + b_{2,i,j}^n \Delta t - \beta_{i,j}^n \sigma_{2,i,j}^n \sqrt{\Delta t} \right).
		\end{aligned}
	\end{equation}
	
	For each proposed branch point \((x_k^*, y_k^*)\), the final branch point \((X_k^h, Y_k^h)\) is computed according to the following rules:
	
	\begin{equation} \label{eq:x-reflection}
		X_k^h = 
		\begin{cases}
			2x_0 - x_k^*, & x_k^* < x_0, \\
			2x_{M_1} - x_k^*, & x_k^* > x_{M_1}, \\
			x_k^*, & \text{otherwise},
		\end{cases}
	\end{equation}
	
	\begin{equation} \label{eq:y-reflection}
		Y_k^h = 
		\begin{cases}
			2y_0 - y_k^*, & y_k^* < y_0, \\
			2y_{M_2} - y_k^*, & y_k^* > y_{M_2}, \\
			y_k^*, & \text{otherwise}.
		\end{cases}
	\end{equation}
	
	For boundary points, we handle them as follows:
	\[
	\begin{aligned}
		f_{h,0,j}^n &= f_{h,1,j}^n,      & 0\le j\le M_2,\\
		f_{h,M_1,j}^n &= f_{h,M_1-1,j}^n, & 0\le j\le M_2,\\
		f_{h,i,0}^n &= f_{h,i,1}^n,      & 0\le i\le M_1,\\
		f_{h,i,M_2}^n &= f_{h,i,M_2-1}^n, & 0\le i\le M_1.
	\end{aligned}
	\]
	
	\medskip
	\noindent
	\textbf{Algorithm 3(Implementation of Reflective Boundary Scheme):}
	\begin{enumerate}
		\item \textbf{Initialization.} Set \(f_{h,i,j}^N = \varphi(x_i,y_j)\) for all grid nodes.
		\item For \(n = N-1,\ldots,0\):
		\begin{enumerate}
			\item For \(i = 1,\ldots,M_1-1\), \(j = 1,\ldots,M_2-1\):
			\begin{enumerate}
				\item For \(k = 1,\ldots,4\):
				\begin{enumerate}
					\item Compute \((x_k^*,y_k^*)\) using \eqref{pr}, then obtain \((X_k^h,Y_k^h)\) via \eqref{eq:x-reflection} and \eqref{eq:y-reflection}.
					\item Evaluate \(U_k = \mathcal{L}^{\mathrm{sp}}f_h(X_k^h,Y_k^h,t_{n+1})\).
				\end{enumerate}
				\item Update \(f_{h,i,j}^n = \dfrac{1}{4(1+r_{i,j}^n\Delta t)}\sum_{k=1}^4 U_k\).
			\end{enumerate}
			\item At the boundary, apply
			\[
			\begin{aligned}
				f_{h,0,j}^n &= f_{h,1,j}^n,      & 0\le j\le M_2,\\
				f_{h,M_1,j}^n &= f_{h,M_1-1,j}^n, & 0\le j\le M_2,\\
				f_{h,i,0}^n &= f_{h,i,1}^n,      & 0\le i\le M_1,\\
				f_{h,i,M_2}^n &= f_{h,i,M_2-1}^n, & 0\le i\le M_1.
			\end{aligned}
			\]
		\end{enumerate}
	\end{enumerate}
	
	Since each update in \eqref{eq:interior-update} forms a nonnegative linear combination of nonnegative values with the sum of coefficients equal to $\frac{1}{1+r^{n}_{i,j}\Delta t} \leq 1$, the scheme maintains both positivity and $L^{\infty}$-stability.
	
	\begin{figure}[htbp]
		\centering
		\includegraphics[width=0.8\textwidth]{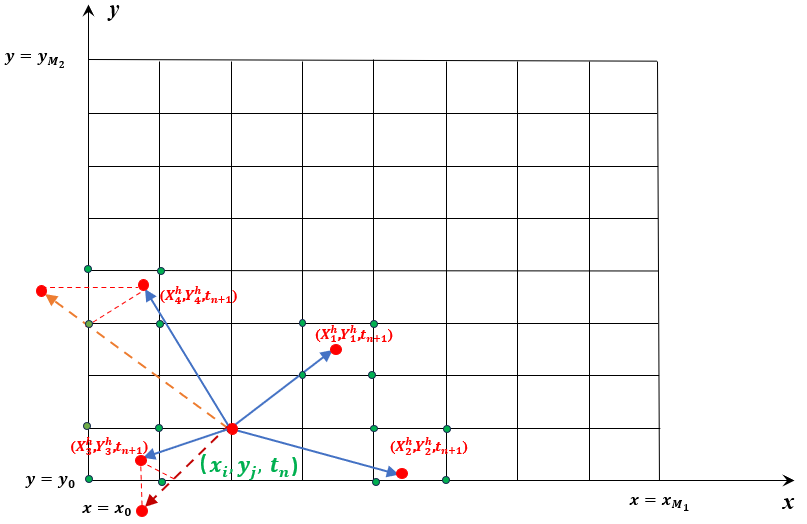}
		\caption{Illustration of Algorithm 3. Here, $X_4^h < x_0$ and $Y_3^h < y_0$ are reflected back into the domain.}
		\label{fig3}
	\end{figure}
	
	\begin{theorem}
		\label{pw2}
		Assume that the exact solution $f$ is sufficiently smooth and $\Delta t = h_1 = h_2$. Then the scheme \eqref{eq:interior-update} satisfies
		\[
		\|e^n\|_\infty \le O( \Delta t^{1/2}).
		\]
	\end{theorem}
	\begin{proof}
		Let $e_{i,j}^n = f_{h,i,j}^n - f(x_i,y_j,t_n)$. We first consider interior nodes $i = M_1 - 1, \ldots, 1$ and $j = M_2 - 1, \ldots, 1$. 
		
		To illustrate the treatment of a reflected endpoint, we take the first branch as a representative example. In the configuration considered here only the $x$–component of the first branch crosses the vertical boundary $x=x_{M_1}$ and is reflected, while the $y$–component remains inside the domain; hence
		\[
		X_1^h = 2x_{M_1}-x_i - b_{1,i,j}^n\Delta t - \alpha_{i,j}^n\sigma_{1,i,j}^n\sqrt{\Delta t},\qquad
		Y_1^h = y_j + b_{2,i,j}^n\Delta t + \alpha_{i,j}^n\sigma_{2,i,j}^n\sqrt{\Delta t}.
		\]
		The remaining branches either stay inside the domain or are reflected (in $x$ or $y$) in an analogous manner, and their Taylor expansions follow the same procedure.
		
		We begin with an error decomposition. From the scheme~\eqref{eq:interior-update},
		\[
		|e_{i,j}^n|
		\le |I_1| + \|e^{n+1}\|_\infty + O(\Delta t^2),\quad i = M_1 - 1, \ldots, 1\quad \text{and}\quad j = M_2 - 1, \ldots, 1,
		\]
		where
		\[
		I_1=\frac{1}{4(1+r_{i,j}^n\Delta t)}
		\sum_{k=1}^4 f(X_k^h,Y_k^h,t_{n+1}) - f(x_i,y_j,t_n).
		\]

		Since the reflection in $x$ enforces the condition
		\[
		x_i + b_{1,i,j}^n\Delta t + \alpha_{i,j}^n\sigma_{1,i,j}^n\sqrt{\Delta t} \ge x_{M_1},
		\]
		we derive the scaling relation
		\begin{equation}
			\label{eq:x-scaling-clean}
			x_{M_1}-x_i = O(\sqrt{\Delta t}).
		\end{equation}
		
		Next, we perform a Taylor expansion and leverage \eqref{eq:x-scaling-clean} for subsequent derivations. For $f_{i,j}^n=f(x_i,y_j,t_n)$, a Taylor series expansion of $f(X_1^h,Y_1^h,t_{n+1})$ gives
		\begin{equation}
			\label{eq:taylor-expanded-clean}
			f(X_1^h,Y_1^h,t_{n+1})
			=
			f_{i,j}^n
			+\partial_t f_{i,j}^n\Delta t
			+\sum_{p=1}^4 K_p+O(\Delta t^{3/2}),
		\end{equation}
		where
		\begin{equation} \label{kp}
			\begin{aligned}
				K_1 &= 2(x_{M_1} - x_i) \left[
				\partial_x f_{i,j}^n
				+ \partial_{xy} f_{i,j}^n (b_{2,i,j}^n\Delta t + \alpha_{i,j}^n\sigma_{2,i,j}^n\sqrt{\Delta t})
				+ (x_{M_1} - x_i) \partial_{xx} f_{i,j}^n
				\right], \\
				K_2 &= \alpha_{i,j}^n \sigma_{1,i,j}^n \sqrt{\Delta t} \left[
				-2\partial_{xx} f_{i,j}^n(x_{M_1}-x_i)
				- \partial_x f_{i,j}^n
				- \partial_{xy} f_{i,j}^n(b_{2,i,j}^n\Delta t + \alpha_{i,j}^n\sigma_{2,i,j}^n\sqrt{\Delta t})
				\right], \\
				K_3 &= b_{1,i,j}^n\Delta t\left[
				-2\partial_{xx}f_{i,j}^n(x_{M_1}-x_i)
				- \partial_x f_{i,j}^n
				\right]
				+ \partial_y f_{i,j}^n(b_{2,i,j}^n\Delta t + \alpha_{i,j}^n\sigma_{2,i,j}^n\sqrt{\Delta t}), \\
				K_4 &= \tfrac12 \partial_{xx} f_{i,j}^n(\sigma_{1,i,j}^n)^2(\alpha_{i,j}^n)^2\Delta t
				+ \tfrac12 \partial_{yy} f_{i,j}^n(\sigma_{2,i,j}^n)^2(\alpha_{i,j}^n)^2\Delta t.
			\end{aligned}
		\end{equation}

		We now examine boundary expansions and Neumann conditions. A Taylor expansion at $x_{M_1}$ gives
		\[
		\begin{aligned}
			\partial_x f_{M_1,j}^n
			&= \partial_x f_{i,j}^n
			+\partial_{xx}f_{i,j}^n(x_{M_1}-x_i)
			+O(\Delta t), \\
			\partial_x f(x_{M_1}, y_j^*, t_n)
			&= \partial_x f_{i,j}^n
			+ \partial_{xy} f_{i,j}^n\Delta y
			+ \partial_{xx} f_{i,j}^n(x_{M_1}-x_i)
			+ O(\Delta t),
		\end{aligned}
		\]
		where
		\[
		y_j^* = y_j + b_{2,i,j}^n\Delta t + \alpha_{i,j}^n\sigma_{2,i,j}^n\sqrt{\Delta t},
		\qquad
		\Delta y = b_{2,i,j}^n\Delta t + \alpha_{i,j}^n\sigma_{2,i,j}^n\sqrt{\Delta t}.
		\]
		
		Applying the Neumann condition, we derive the key relations
		\begin{equation}
			\label{eq:key-boundary-clean}
			\begin{aligned}
				\partial_x f_{i,j}^n 
				+ \partial_{xy} f_{i,j}^n\Delta y
				+ \partial_{xx} f_{i,j}^n(x_{M_1}-x_i)
				&= O(\Delta t),\\
				-\partial_{xx} f_{i,j}^n(x_{M_1}-x_i)
				&= \partial_x f_{i,j}^n + O(\Delta t).
			\end{aligned}
		\end{equation}
		
		Substituting \eqref{eq:x-scaling-clean} and \eqref{eq:key-boundary-clean} into \eqref{eq:taylor-expanded-clean} yields
		\[
		\begin{aligned}
			f(X_1^h,Y_1^h,t_{n+1})
			=&\ f_{i,j}^n 
			+ \partial_t f_{i,j}^n\Delta t
			+ \partial_x f_{i,j}^n(b_{1,i,j}^n\Delta t + \alpha_{i,j}^n\sigma_{1,i,j}^n\sqrt{\Delta t}) \\
			&+ \partial_y f_{i,j}^n(b_{2,i,j}^n\Delta t + \alpha_{i,j}^n\sigma_{2,i,j}^n\sqrt{\Delta t}) \\
			&+ \partial_{xy}f_{i,j}^n(\alpha_{i,j}^n)^2\sigma_{1,i,j}^n\sigma_{2,i,j}^n\Delta t \\
			&+ \frac12 \partial_{xx}f_{i,j}^n(\alpha_{i,j}^n)^2(\sigma_{1,i,j}^n)^2\Delta t
			+ \frac12 \partial_{yy}f_{i,j}^n(\alpha_{i,j}^n)^2(\sigma_{2,i,j}^n)^2\Delta t
			+ O(\Delta t^{3/2}).
		\end{aligned}
		\]
		Repeating the same expansion for $k=2,3,4$ shows that
		\[
		|I_1| = O(\Delta t^{3/2}).
		\]
		Thus for interior nodes, we conclude that
		\[
		|e_{i,j}^n|
		\le \|e^{n+1}\|_\infty + O(\Delta t^{3/2}).
		\]
		
		Turning to boundary nodes, we consider the representative case $i=M_1,1\leq j\leq M_2-1$; all other boundary nodes are handled similarly:
		\[
		|e_{M_1,j}^n|
		\le 
		|f_{h,M_1-1,j}^n - f(x_{M_1-1},y_j,t_n)|
		+
		|f(x_{M_1-1},y_j,t_n)-f(x_{M_1},y_j,t_n)|
		=:T_1+T_2.
		\]
		The first term satisfies
		\[
		T_1 \le \|e^{n+1}\|_\infty + O(\Delta t^{3/2}),
		\]
		and by a Taylor expansion plus Neumann condition,
		\[
		T_2 = O(h_1^2) = O(\Delta t^2).
		\]
		Hence
		\[
		|e_{M_1,j}^n|
		\le \|e^{n+1}\|_\infty + O(\Delta t^{3/2}),
		\]
		and the same estimate holds on all boundaries.
		
		Finally, we establish the global estimate. With $\Delta t=h_1=h_2$,
		\[
		\|e^n\|_\infty \le \|e^{n+1}\|_\infty + O(\Delta t^{3/2}).
		\]
		Iterating from $e^N=0$ gives
		\[
		\|e^n\|_\infty \le O(N\Delta t^{3/2}) = O(\Delta t^{1/2}),
		\]
		which completes the proof.
	\end{proof}
	\subsection{Periodic Boundary Conditions}
	\label{subsec:periodic}
	
	We now consider periodic boundary conditions.
	Let $\Omega = (x_0,x_{M_1}) \times (y_0,y_{M_2})$ denote the computational domain, and define the spatial periods
	\[
	L_x := x_{M_1}-x_0, \qquad L_y := y_{M_2}-y_0.
	\]
	The solution $f$ and all PDE coefficients are assumed to be defined on $\mathbb{R}^2$ and to be doubly periodic with respect to these periods.  
	That is, for all $(x,y,t)\in\mathbb{R}^2\times[0,T]$ and all integers $m,n$,
	\[
	f(x + mL_x,\, y,\, t) = f(x,y,t), 
	\qquad
	f(x,\, y + nL_y,\, t) = f(x,y,t),
	\]
	and the same periodicity holds for $A$, $B$, and $r$ in \eqref{eq:PDE-periodic}.  
	The terminal data $\varphi$ satisfies the same periodicity, and all coefficient assumptions and notation remain the same as in the previous subsections.
	
	Under periodicity, the PDE is naturally posed on the 2D torus 
	$\mathbb{T}^2 := \mathbb{R}^2 / (L_x\mathbb{Z} \times L_y\mathbb{Z})$:
	
	\begin{equation}  \label{eq:PDE-periodic}
		\begin{cases}  
			\partial_t f + \frac{1}{2}\mathrm{Tr}(AA^\top D^2 f) + B\cdot\nabla f - r(x,y,t)f = 0, 
			& (x,y)\in\mathbb{R}^2,\ t\in[0,T), \\
			f(x,y,T) = \varphi(x,y), & (x,y)\in\mathbb{R}^2
		\end{cases}
	\end{equation}
	with $f$ understood as a periodic extension on $\mathbb{R}^2$.
	
	In the discrete approximation, any branch endpoint that leaves $\Omega$ must be wrapped 
	back into the fundamental cell via periodicity.  
	Define the wrapping operator $\mathcal{W}: \mathbb{R}^2\to\Omega$ by
	\begin{equation}
		\label{wrap}
		\mathcal{W}(x,y)
		=
		\bigl(x_0 + \{(x-x_0)\bmod L_x\},\ 
		y_0 + \{(y-y_0)\bmod L_y\}\bigr),
	\end{equation}
	For each grid node $(x_i,y_j,t_n)$, compute the four  endpoints by \eqref{pr}
	\[
	(x_k^*, y_k^*), \qquad k=1,\dots,4,
	\]
	and then apply periodic wrapping:
	\[
	(X_k^h,Y_k^h) = \mathcal{W}(x_k^*, y_k^*), \qquad k=1,\dots,4.
	\]
	Because periodicity eliminates boundary stopping, all branches share the next discrete time level 
	$t_{n+1}$, and the one-step update becomes
	\begin{equation}
		\label{eq:periodic-update}
		f_{h,i,j}^n
		=
		\frac{1}{4(1+r_{i,j}^n\Delta t)}
		\sum_{k=1}^4 
		\mathcal{L}^{\mathrm{sp}} f_h(X_k^h,Y_k^h, t_{n+1}).
	\end{equation}
	
	\noindent
	\textbf{Algorithm 4 (Periodic Boundary Scheme for \eqref{eq:PDE-periodic}):}
	\begin{enumerate}
		\item \textbf{Initialization.} Set \(f_{h,i,j}^N = \varphi(x_i, y_j)\) for all grid nodes.
		\item For \(n = N-1, \ldots, 0\):
		\begin{enumerate}
			\item For \(i = 0, \ldots, M_1,\; j = 0, \ldots, M_2\): 
			\begin{enumerate}
				\item For \(k = 1, \ldots, 4\):
				\begin{enumerate}
					\item Compute proposed position \((x_k^*, y_k^*)\) by \eqref{pr} at \(t_{n+1}\).
					\item Wrap to \((X_k^h,Y_k^h) = \mathcal{W}(x_k^*, y_k^*)\) by \eqref{wrap}.
					\item Evaluate \(U_k = \mathcal{L}^{\mathrm{sp}} f_h(X_k^h,Y_k^h, t_{n+1})\).
				\end{enumerate}
				\item Update \(f_{h,i,j}^n = \dfrac{1}{4(1 + r_{i,j}^n \Delta t)} \sum_{k=1}^4 U_k\).
			\end{enumerate}
		\end{enumerate}
	\end{enumerate}
	
	\begin{figure}[htbp]
		\centering
		\includegraphics[width=0.8\textwidth]{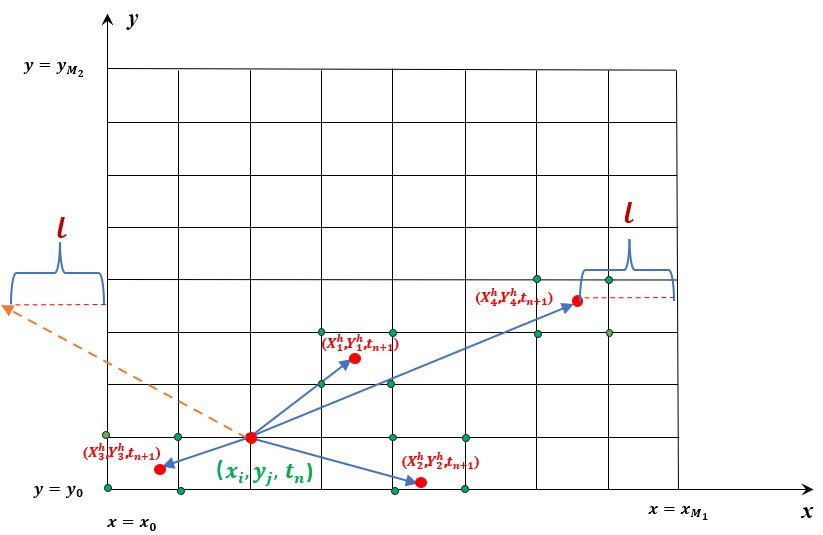}
		\caption{Illustration of Algorithm 4. Here, the periodic condition is applied to \(X_4^h\).}
		\label{fig4}
	\end{figure}
	
	Since each update in \eqref{eq:periodic-update} forms a nonnegative linear combination of periodically interpolated values, with the sum of coefficients equal to $\frac{1}{1+r^{n}_{i,j}\Delta t} \leq 1$, the scheme preserves both linearity and positivity, and is uniformly stable in the \(L^\infty\) norm.
	
	\newpage
	
	\begin{theorem}[Improved Convergence under Periodicity]
		\label{thm:periodic-convergence}
		Assume the exact solution \(f\) is smooth and spatially periodic on \(\Omega\). With \(\Delta t = h_1 = h_2\), there exists a constant \(C > 0\), independent of mesh sizes, such that the numerical solution from Algorithm 3 satisfies
		\[
		\|e^n\|_\infty \le O( \Delta t).
		\]
	\end{theorem}
	
	\begin{proof}[Proof sketch]
		Under periodic boundary conditions, branches never encounter boundary stopping, ensuring all discrete stopping times equal \(t_{n+1}\). Consequently, the local quadrature remainder \(I_1\)—previously \(O(\Delta t^{3/2})\) due to asymmetric partial stops—now reduces to higher order: Taylor expansion of all four branch contributions around \((x_i, y_j, t_n)\) shows exact cancellation of odd spatial moments and matching of second-order moments with diffusion coefficients, yielding quadrature error \(O(\Delta t^2)\). Combined with spatial interpolation error \(O(h_1^2 + h_2^2)\), the local truncation error becomes:
		\[
		R_{\mathrm{loc}} = O(\Delta t^2) + O(h_1^2 + h_2^2).
		\]
		Global error scales as \(R_{\mathrm{loc}}/\Delta t\), so with \(\Delta t = h_1 = h_2\), we obtain first-order global error \(O(\Delta t)\). 
	\end{proof}
	\section{Discussion}
	\label{sem}
	Consider the linear parabolic equation in \eqref{eq:PDE-backward} with the diffusion matrix  
	\[
	AA^{\top}= \begin{pmatrix}
		\sigma_1^2 & \rho\sigma_1\sigma_2 \\[2pt]
		\rho\sigma_1\sigma_2 & \sigma_2^2
	\end{pmatrix},
	\]  
	where \(A=(A_1,\dots,A_P)\) is a \(2\times P\) matrix whose columns are given by \(A_i=(A_{i_1},A_{i_2})^{\!\top}\) for \(i=1,\dots,P\).
	
	A common approach to discretizing such problems is the semi-Lagrangian (SL) method{\cite{bib4}}. Using the fifth type of approximation from section 5.1 of the original text, the advection and diffusion terms can be discretized as
	\begin{equation}
		\begin{aligned}
			L_k[\phi](t, x, y) =& \sum_{i=1}^{P-1} \frac{\phi(\Delta x_{i_1}, \Delta y_{i_2}) - 2\phi(x, y) + \phi(\Delta x_{-i_1}, \Delta y_{-i_2})}{2k^2}\\
			& + \frac{\phi(\Delta x_{P_1}, \Delta y_{P_2}) - 2\phi(x, y) + \phi(\Delta x_{-P_1}, \Delta y_{-P_2})}{2k^2},
		\end{aligned}
	\end{equation}
	where
	\[
	\begin{aligned}
		\Delta x_{i_1} &= x + k A_{i_1}, & \Delta y_{i_2} &= y + k A_{i_2},\\
		\Delta x_{-i_1} &= x - k A_{i_1}, & \Delta y_{-i_2} &= y - k A_{i_2},\\
		\Delta x_{P_1} &= x + k A_{P_1} + k^2 b_1, & \Delta y_{P_2} &= y + k A_{P_2} + k^2 b_2,\\
		\Delta x_{-P_1} &= x - k A_{P_1} + k^2 b_1, & \Delta y_{-P_2} &= y - k A_{P_2} + k^2 b_2.
	\end{aligned}
	\]
	Combining this approximation with the spatial grid and applying an interpolation operator \(\mathcal{L}\)(linear in the LISL variant, or higher‑order and solution‑dependent in nonlinear MPCSL)gives a discrete scheme. Employing a \(\theta\)-scheme in time leads to the complete SL update
	\begin{equation}
		\delta_{\Delta t} U_{i,j}^n + L_k[\mathcal{L} \overline{U}^{\theta,n+1}]_{i,j}^{n+1-\theta} - r_{i,j}^{n+1-\theta} \overline{U}_{i,j}^{\theta,n+1}=0,
	\end{equation}
	where
	\[
	U_{i,j}^n \approx u(t_n, x_i, y_j),\quad \delta_{\Delta t} U_{i,j}^n = \frac{U_{i,j}^{n+1} - U_{i,j}^{n}}{\Delta t},\quad \overline{U}^{\theta,n+1} = (1-\theta)U^{n+1} + \theta U^n.
	\]
	For monotonicity and $L^\infty$ stability, the  LISL scheme  must satisfy the CFL condition
	\[
	(1 - \theta) \Delta t \left[ \frac{P}{k^2} + r_{i,j}^{n+1-\theta} \right] \leq 1, \quad \forall i,j,n.
	\]
	
	When no values outside the computational domain are required, the error of LISL scheme is
	\[
	O\left( |1 - 2\theta|\Delta t + \Delta t^2 + k^2 + \frac{\Delta h^2}{k^2} \right),
	\]
	from which it can  be seen that spatial convergence typically requires \( k = O(\Delta h^{1/2}) \). From the CFL condition, we know that when \( \theta < 1 \), \( \Delta t \) must satisfy \( \Delta t = O(k^2) = O(h) \). Although this is less restrictive than the \( \Delta t = O(\Delta h^2) \) limitation of classical explicit parabolic schemes, it still imposes a conditional stability constraint.
	Moreover, a major practical challenge arises near boundaries: due to its wide stencil, the Semi-Lagrangian (SL) discretization often extends beyond the computational domain \( \Omega \). For Dirichlet boundary conditions, this requires extrapolating values from outside the domain, which can degrade accuracy and potentially violate monotonicity near the boundary.
	.

	The expectation‑based method proposed in this work differs from the SL approach in three fundamental aspects:
	\begin{enumerate}
		\item \textbf{Different Boundary Handling:} The SL method requires boundary extrapolation, which often compromises not only precision but also the preservation of monotonicity. In contrast, our method handles boundaries natively, preserving both accuracy and monotonicity without the need for extrapolation.
		
		\item \textbf{Stability and Monotonicity Conditions.} The explicit LISL method (with $\theta=0$) requires a CFL condition to ensure both monotonicity and $L^{\infty}$ stability. In contrast, our method, which is derived from the Feynman-Kac formula, is unconditionally $L^{\infty}$ stable and naturally preserves monotonicity. It is noteworthy that both methods adopt the scaling $\mathcal{O}(\Delta t) = \mathcal{O}(h)$: for the LISL method, this scaling is necessary to satisfy stability and consistency requirements, while for our method, it is adopted primarily to maintain consistency.
		
		\item \textbf{Different Truncation Errors:} Away from the boundaries, the truncation error of our method is $O(\Delta t)$. In contrast, the truncation error of the LISL method away from the boundaries is $O\left( |1 - 2\theta|\Delta t + \Delta t^2 + k^2 + \frac{\Delta h^2}{k^2} \right)$.
	\end{enumerate}

	Overall, compared to the SL approach, our method offers a more robust framework:
	it avoids accuracy and monotonicity loss at boundaries by design, and it removes the
	CFL constraint for monotonicity and $L^{\infty}$ stable through its Feynman-Kac foundation, resulting in
	an unconditionally stable and monotone method.
	\section{Numerical Results}
	\label{sec4}
	In this section, we present numerical experiments to verify the convergence properties of the schemes proposed in Section~\ref{sec2}. We report the maximum error and $L^2$  error at the initial time $t=0$:
	\[
	\text{Error}_{L^\infty}
	= \max_{0\le i\le M_1}\max_{0\le j\le M_2} 
	\big| f_{h,i,j}^0 - f(x_i,y_j,0) \big|,
	\]
	\[
	\text{Error}_{L^2}
	= \left( h_1 h_2 \sum_{i=0}^{M_1}\sum_{j=0}^{M_2}
	\big| f_{h,i,j}^{0} - f(x_i,y_j,0) \big|^2 \right)^{1/2}.
	\]
	
	The empirical convergence rate is measured under successive mesh refinement:
	\[
	\text{Rate}=
	\frac{\log\!\left( \text{Error}(h_1) / \text{Error}(h_2) \right)}
	{\log(h_1/h_2)},
	\]
	where $\text{Error}(h)$ denotes the error corresponding to the spatial step size $h$.

	We begin with Dirichlet boundary conditions and compare  Algorithm 1 and 2  with the explicit ($\theta=0$) LISL method described in Section~\ref{sem}, using two boundary treatments: Exact and Extrapolation. We subsequently examine performance under homogeneous Neumann and periodic boundary conditions.
	\subsection{Dirichlet Boundary Conditions}
	
	We solve the backward parabolic problem in $\Omega=(0,1)^2$ with terminal time $T=1$:
	\[i
	\begin{cases}
		\partial_t f
		+\dfrac12\sigma_1^2 f_{xx}
		+\dfrac12\sigma_2^2 f_{yy}
		+\rho\sigma_1\sigma_2 f_{xy}
		- r f = 0,
		& (x,y)\in\Omega,\; t\in[0,1),\\[4pt]
		f(x,y,1)=e^{1+x+y},&(x,y)\in\Omega\\[4pt]
		f|_{\partial\Omega}(x,y,t)=e^{t+x+y}, &t\in[0,1),
	\end{cases}
	\]
	where
	\[
	\sigma_1^2=x^2y^2t^2,\qquad
	\sigma_2^2=(xyt+1)^2,\qquad
	\rho=1,\qquad
	r=1+\dfrac12x^2y^2t^2+\dfrac12(xyt+1)^2+xyt(xyt+1).
	\]
	
	The exact solution is $f(x,y,t)=e^{t+x+y}$.  
	The associated covariance matrix 
	\[
	AA^\top=\begin{pmatrix}
		x^2y^2t^2 & xyt(xyt+1)\\
		xyt(xyt+1) & (xyt+1)^2
	\end{pmatrix}
	\]
	is not diagonally dominant, placing the model in the class of mixed-derivative anisotropic diffusion problems for which compact positivity-preserving nine-point schemes typically fail. This motivates our non-compact expectation-based approach.

	We compare Algorithms 1 and 2 with the explicit semi-Lagrangian (LISL) scheme of Section~\ref{sem}, which employs directional splitting and linear interpolation for off-grid evaluations. For stability, we use $\Delta t = h/4$.
	
	\textbf{Boundary treatments.}
	Due to stencil non-compactness, LISL evaluation points may fall outside $\Omega$. We test:
	
	\begin{itemize}
		\item \textbf{Exact}:  
		Outside points are assigned the exact value $e^{t+x+y}$.  
		This can achieve the theoretical convergence order but is not practical.
		
		\item \textbf{Extrapolation}:  
		For a point $(x_q, y_q)$ outside $\Omega$, we first compute  indices:
		$i_0 = (x_q - x_{0})/h_1 + 1$ and $j_0 = (y_q - y_{0})/h_2 + 1$. 
		Then we clamp the indices to the nearest boundary cell by setting 
		$i = \min(\max(\lfloor i_0 \rfloor, 1), M_1-1)$ and $j = \min(\max(\lfloor j_0 \rfloor, 1), M_2-1)$ interpolation at $(x_q, y_q)$ using the four grid nodes 
		$(x_i, y_j), (x_{i+1}, y_j), (x_i, y_{j+1}), (x_{i+1}, y_{j+1})$. 
	\end{itemize}

	Algorithms 1 and 2 use $\Delta t=h_1=h_2$.  
	The LISL method uses $h_1=h_2=h$,$\Delta t=h/4$.  
	Errors at $t=0$ are summarized in Tables~\ref{tab:dirichlet_explicit}--\ref{tab:lisl_extrapolation}.
	
	\begin{table}[htb]
		\centering
		\caption{Algorithm 1: $L^\infty$ and $L^2$ errors at $t=0$ and convergence rates for the Dirichlet problem.}
		\label{tab:dirichlet_explicit}
		\small
		\begin{tabular}{ccccccc}
			\hline
			$M_1$ & $M_2$ & $N$ & $\text{Error}_{L^{\infty}}$ & $\text{Rate}_{L^{\infty}}$ & $\text{Error}_{L^{2}}$ & $\text{Rate}_{L^{2}}$ \\
			\hline
			20  & 20  & 20  & $1.1733 \times 10^{-1}$ & ---    & $6.9700 \times 10^{-2}$ & ---    \\
			40  & 40  & 40  & $6.0217 \times 10^{-2}$ & 0.9623 & $3.5784 \times 10^{-2}$ & 0.9618 \\
			80  & 80  & 80  & $3.0549 \times 10^{-2}$ & 0.9791 & $1.8132 \times 10^{-2}$ & 0.9807 \\
			160 & 160 & 160 & $1.5381 \times 10^{-2}$ & 0.9900 & $9.1251 \times 10^{-3}$ & 0.9907 \\
			320 & 320 & 320 & $7.7104 \times 10^{-3}$ & 0.9963 & $4.5758 \times 10^{-3}$ & 0.9958 \\
			640 & 640 & 640 & $3.8571 \times 10^{-3}$ & 0.9993 & $2.2905 \times 10^{-3}$ & 0.9984 \\
			\hline
		\end{tabular}
	\end{table}
	
	\begin{table}[htb]
		\centering
		\caption{ Algorithm 2: $L^\infty$ and $L^2$ errors at $t=0$ and convergence rates for the Dirichlet problem.}
		\label{tab:dirichlet_symmetric}
		\small
		\begin{tabular}{ccccccc}
			\hline
			$M_1$ & $M_2$ & $N$ & $\text{Error}_{L^{\infty}}$ & $\text{Rate}_{L^{\infty}}$ & $\text{Error}_{L^{2}}$ & $\text{Rate}_{L^{2}}$ \\
			\hline
			20  & 20  & 20  & $1.1084 \times 10^{-1}$ & ---    & $6.8080 \times 10^{-2}$ & ---    \\
			40  & 40  & 40  & $5.7351 \times 10^{-2}$ & 0.9506 & $3.5087 \times 10^{-2}$ & 0.9563 \\
			80  & 80  & 80  & $2.9344 \times 10^{-2}$ & 0.9668 & $1.7851 \times 10^{-2}$ & 0.9749 \\
			160 & 160 & 160 & $1.4901 \times 10^{-2}$ & 0.9777 & $9.0165 \times 10^{-3}$ & 0.9854 \\
			320 & 320 & 320 & $7.5254 \times 10^{-3}$ & 0.9855 & $4.5352 \times 10^{-3}$ & 0.9914 \\
			640 & 640 & 640 & $3.7876 \times 10^{-3}$ & 0.9905 & $2.2756 \times 10^{-3}$ & 0.9949 \\
			\hline
		\end{tabular}
	\end{table}
	
	\begin{table}[htb]
		\centering
		\caption{LISL-Exact scheme: $L^\infty$ errors at $t=0$ and convergence rates for the Dirichlet problem.  
			\emph{Note:} Explicit method ($\theta=0$), spatial step $h=1/M$, time step $\Delta t=h/4$.}
		\label{tab:lisl_exact}
		\small
		\begin{tabular}{ccccc}
			\hline
			$M$ & $h$ & $\Delta t$ & $\text{Error}_{L^{\infty}}$ & $\text{Rate}_{L^{\infty}}$ \\
			\hline
			20  & 0.05       & 0.0125      & $5.4703 \times 10^{-2}$ & ---    \\
			40  & 0.025      & 0.00625     & $2.8070 \times 10^{-2}$ & 0.9626 \\
			80  & 0.0125     & 0.003125    & $1.4617 \times 10^{-2}$ & 0.9414 \\
			160 & 0.00625    & 0.0015625   & $7.3556 \times 10^{-3}$ & 0.9908 \\
			320 & 0.003125   & 0.00078125  & $3.6808 \times 10^{-3}$ & 0.9988 \\
			640 & 0.0015625  & 0.000390625 & $1.8327 \times 10^{-3}$ & 1.0060 \\
			\hline
		\end{tabular}
	\end{table}
	
	\begin{table}[htb]
		\centering
		\caption{LISL-Extrapolation scheme: $L^\infty$ errors at $t=0$ and convergence rates for the Dirichlet problem.  
			\emph{Note:} Explicit method ($\theta=0$), spatial step $h=1/M$, time step $\Delta t=h/4$.}
		\label{tab:lisl_extrapolation}
		\small
		\begin{tabular}{ccccc}
			\hline
			$M$ & $h$ & $\Delta t$ & $\text{Error}_{L^{\infty}}$ & $\text{Rate}_{L^{\infty}}$ \\
			\hline
			20  & 0.05       & 0.0125      & $8.695 \times 10^{-1}$ & ---    \\
			40  & 0.025      & 0.00625     & $6.773 \times 10^{-1}$ & 0.360  \\
			80  & 0.0125     & 0.003125    & $5.338 \times 10^{-1}$ & 0.343  \\
			160 & 0.00625    & 0.0015625   & $4.039 \times 10^{-1}$ & 0.402  \\
			320 & 0.003125   & 0.00078125  & $3.045 \times 10^{-1}$ & 0.408  \\
			640 & 0.0015625  & 0.000390625 & $2.214 \times 10^{-1}$ & 0.460  \\
			\hline
		\end{tabular}
	\end{table}

	\newpage

	The results presented in Tables~\ref{tab:dirichlet_explicit}--\ref{tab:lisl_extrapolation} demonstrate the following:
	
	\textbf{1. Convergence performance.}
	Algorithms 1 and 2 both exhibit near first-order convergence. Algorithm 2 aligns with the theoretical prediction of Theorem~\ref{thm:symmetric-conv}, while Algorithm 1 performs better than the conservative $O(\Delta t^{1/2})$ estimate of Theorem~\ref{thm:compact-conv}. This improved behavior occurs because, for most interior points, all branches share the same stopping time $t_{n+1}$, allowing lower-order error terms to cancel.
	
	In comparison, the LISL--Exact method also achieves near first-order accuracy, which is consistent with theoretical expectations. However, it relies on knowing the exact solution outside the domain and is therefore impractical. The LISL--Extrapolation variant, which uses bilinear extrapolation at the boundary, performs markedly worse, with convergence rates around $0.4$---clearly illustrating how  extrapolation  degrades the accuracy of semi-Lagrangian schemes.
	
	\textbf{2. Boundary treatment.}
	Algorithms 1 and 2 naturally incorporate boundary interactions through their stopping time mechanism and the adaptive probabilities defined in \eqref{eq:omega}. Consequently, they require neither extrapolation nor exact exterior values, and they maintain  accuracy and pomonotonicity  up to the boundary.
	
	By contrast, the semi-Lagrangian method faces an inherent trade-off: it either depends on impractical exact boundary data or suffers from poor convergence when extrapolation is employed. This limitation underscores a structural difficulty of LISL-type schemes in handling boundaries accurately.
	\medskip
	
	In summary, the expectation-based schemes proposed here display much more robust boundary handling than the semi-Lagrangian approach. By avoiding both extrapolation and the need for exact boundary information, they offer a more reliable and practical choice for anisotropic diffusion problems.

	\subsection{Parabolic Problem with Second-Type (Homogeneous Neumann) Boundary Conditions}
	\label{subsec:numerics-reflective}
	
	We test the reflective (Neumann) implementation on the exact solution
	\[
	f(x,y,t) = e^{t} \sin\left(\pi x - \frac{\pi}{2}\right) \sin\left(\pi y - \frac{\pi}{2}\right), \qquad (x,y) \in [0,1]^2, \; t \in [0,1],
	\]
	solving backward the PDE:
	\begin{equation*} 
		\label{eq:exp-reflective}
		\begin{cases}
			\displaystyle
			\partial_t f + \frac{1}{2} \sigma_1^2 \partial_{xx} f + \frac{1}{2} \sigma_2^2 \partial_{yy} f + \rho \sigma_1 \sigma_2 \partial_{xy} f  - r f = 0,&(x,y)\in\Omega,t\in[0,1) \\[6pt]
			f(x,y,1) = e \sin\left(\pi x - \frac{\pi}{2}\right) \sin\left(\pi y - \frac{\pi}{2}\right),&(x,y)\in\Omega \\[4pt]
			\dfrac{\partial f}{\partial n}=0 ,&(x,y)\in\partial\Omega,,t\in[0,1),
		\end{cases}
	\end{equation*}
	where $\Omega=(0,1)^2$ and the coefficients are defined as:
	\[
	\begin{aligned}
		\sigma_1^2 &= \dfrac{x^2}{4\pi^2}, \\
		\sigma_2^2 &= \dfrac{\sin(\pi x - \frac{\pi}{2})^4 \sin(\pi y - \frac{\pi}{2})^4}{4\pi^2}, \\
		\rho &= 1, \\
		r &= 1 - \frac{1}{8}x^2- \frac{1}{8}\sin^4\left(\pi x - \frac{\pi}{2}\right) \sin^4\left(\pi y - \frac{\pi}{2}\right)\\&\qquad + \frac{1}{4}x \sin\left(\pi x - \frac{\pi}{2}\right) \sin\left(\pi y - \frac{\pi}{2}\right) \cos\left(\pi x - \frac{\pi}{2}\right) \cos\left(\pi y - \frac{\pi}{2}\right).
	\end{aligned}
	\]
	The associated covariance matrix is
	\[
	A A^\top = \
	\begin{pmatrix}
		\dfrac{x^2}{4\pi^2} &\dfrac{x\sin(\pi x - \frac{\pi}{2})^2 \sin(\pi y - \frac{\pi}{2})^2}{4\pi^2} \\[2pt]
		\\
		\dfrac{x\sin(\pi x - \frac{\pi}{2})^2 \sin(\pi y - \frac{\pi}{2})^2}{4\pi^2} &\dfrac{\sin(\pi x - \frac{\pi}{2})^4 \sin(\pi y - \frac{\pi}{2})^4}{4\pi^2}
	\end{pmatrix},
	\]
	Table~\ref{tab:reflective} shows the \(L^\infty\) errors at \(t = 0\) obtained by the reflective algorithm (Algorithm 3).
	
	\begin{table}[htb]
		\centering
		\caption{$L^{\infty}$ error and $L^2$ error with convergence rates at $t=0$ for the reflective Neumann problem (Algorithm 3).}
		\label{tab:reflective}
		\small
		\begin{tabular}{@{}ccccccc@{}}
			\toprule
			\(M_1\) & \(M_2\) & \(N\) &$\text{Error}_{L^\infty}$ & $\text{Rate}_{L^\infty}$ & $\text{Error}_{L^2}$ & $\text{Rate}_{L^2}$ \\
			\midrule
			20  & 20  & 20  & \(8.1891 \times 10^{-2}\) & --- & \(2.2530 \times 10^{-2}\) & --- \\
			40  & 40  & 40  & \(3.5314 \times 10^{-2}\) & 1.2135 & \(8.2544 \times 10^{-3}\) & 1.4486 \\
			80  & 80  & 80  & \(9.7472 \times 10^{-3}\) & 1.8572 & \(3.2139 \times 10^{-3}\) & 1.3608 \\
			160 & 160 & 160 & \(5.0892 \times 10^{-3}\) & 0.9376 & \(1.6119 \times 10^{-3}\) & 0.9956 \\
			320 & 320 & 320 & \(2.1656 \times 10^{-3}\) & 1.2327 & \(6.8456 \times 10^{-4}\) & 1.2355 \\
			640 & 640 & 640 & \(1.0320 \times 10^{-3}\) & 1.0693 & \(3.5296 \times 10^{-4}\) & 0.9557 \\
			\bottomrule
		\end{tabular}
	\end{table}
	
	The observed rates for the reflective test exceed the conservative theoretical bound in some grid regimes; this is consistent with the fact that for most interior nodes the four branches do not hit the boundary and therefore attain higher local accuracy.
	
	\subsection{Periodic Boundary Conditions}
	\label{subsec:numerics-periodic}
	
	The exact solution  
	\[
	f(x, y, t) = e^{t} \sin(2\pi x) \sin(2\pi y)
	\]  
	is defined on $\mathbb{R}^2$, with the computational domain taken as the unit square \([0, 1] \times [0, 1]\). Periodic boundary conditions (with periods \(L_x = L_y = 1\)) are applied, and the time interval is \([0, 1]\).
	which satisfies the backward PDE:
	\begin{equation*} 
		\label{eq:exp-periodic}
		\begin{cases}
			\displaystyle
			\partial_t f + \frac{1}{2} \sigma_1^2 \partial_{xx} f + \frac{1}{2} \sigma_2^2 \partial_{yy} f + \rho \sigma_1 \sigma_2 \partial_{xy} f + b_1 \partial_x f + b_2 \partial_y f - r f = 0,&(x,y)\in\mathbb{R}^2, t\in[0,1) \\[6pt]
			f(x,y,1)=e\sin{(2\pi x)}\sin{(2\pi y)},&(x,y)\in\mathbb{R}^2.
		\end{cases}
	\end{equation*}
	with periodic boundary conditions and coefficients defined as:
	\[
	\begin{aligned}
		&b_1(x,y) = \frac{1}{2\pi}\sin(2\pi x) \cos(2\pi y), \\
		&b_2(x,y) = -\frac{1}{2\pi}\sin(2\pi y) \cos(2\pi x), \\
		&\sigma_1^2(x,y) = \frac{1}{16\pi^2} \cos(2\pi x)^4 \cos(2\pi y)^4, \\
		&\sigma_2^2(x,y) = \frac{1}{16\pi^2} \sin(2\pi x)^4 \sin(2\pi y)^4, \\
		&\rho = 1, \\
		&r(x,y) = 1 -\frac18 \cos^4(2\pi x) \cos^4(2\pi y) -\frac18 \sin^4(2\pi x) \sin^4(2\pi y)\\&\qquad\qquad +\frac14 \sin(2\pi x) \sin(2\pi y) \cos(2\pi x)^3 \cos(2\pi y)^3.
	\end{aligned}
	\]
	The associated covariance matrix is
	\[
	A A^\top =
	\begin{pmatrix}
		\dfrac{1}{16\pi^2} \cos(2\pi x)^4 \cos(2\pi y)^4 & \dfrac{1}{(16\pi)^2} \sin(4\pi x)^2 \sin(4\pi y)^2\\[2pt]
		\\
		\dfrac{1}{(16\pi)^2} \sin(4\pi x)^2 \sin(4\pi y)^2 & \dfrac{1}{16\pi^2} \sin(2\pi x)^4 \sin(2\pi y)^4
	\end{pmatrix},
	\]
	and the  matrix is not diagonally dominant. Thus, this periodic test likewise represents an anisotropic mixed-derivative diffusion case.
	
	The periodic run errors at \(t = 0\) are given in Table~\ref{tab:periodic_consistent}.
	
	\begin{table}[htb]
		\centering
		\caption{$L^{\infty}$ error and $L^2$ error with convergence rates at $t=0$ for the periodic problem (Algorithm 4).}
		\label{tab:periodic_consistent}
		\small
		\begin{tabular}{@{}ccccccc@{}}
			\toprule
			\(M_1\) & \(M_2\) & \(N\) &$\text{Error}_{L^\infty}$ & $\text{Rate}_{L^\infty}$ & $\text{Error}_{L^2}$ & $\text{Rate}_{L^2}$ \\
			\midrule
			20  & 20  & 20  & \(1.3661 \times 10^{-1}\) & --- & \(7.4054 \times 10^{-2}\) & --- \\
			40  & 40  & 40  & \(7.2895 \times 10^{-2}\) & 0.9061 & \(4.1730 \times 10^{-2}\) & 0.8275 \\
			80  & 80  & 80  & \(3.4862 \times 10^{-2}\) & 1.0642 & \(2.0683 \times 10^{-2}\) & 1.0126 \\
			160 & 160 & 160 & \(1.7847 \times 10^{-2}\) & 0.9659 & \(1.0551 \times 10^{-2}\) & 0.9711 \\
			320 & 320 & 320 & \(9.1831 \times 10^{-3}\) & 0.9587 & \(5.4682 \times 10^{-3}\) & 0.9482 \\
			640 & 640 & 640 & \(4.8347 \times 10^{-3}\) & 0.9256 & \(2.8229 \times 10^{-3}\) & 0.9539 \\
			\bottomrule
		\end{tabular}
	\end{table}
	\newpage
	The periodic experiment exhibits empirical convergence close to first order on refined meshes, in agreement with Theorem~\ref{thm:periodic-convergence}.
	
	\section{Conclusion}
	\label{sec5}
	This paper has introduced a novel non-compact numerical framework for linear non-divergence form parabolic equations.Based on the Feynman–Kac formula, the solution is expressed as a conditional expectation of an associated diffusion process. By employing a wide stencil scheme approximation with positivity-preserving interpolation and adaptive probabilities, we obtain a scheme that inherently guarantees positivity and permits the favorable time-step scaling $\Delta t \sim h$, substantially relaxing the $\Delta t = O(h^2)$ restriction typical of explicit compact schemes.
	
	A key contribution lies in the systematic and accurate treatment of boundary conditions---a common challenge for non-compact methods. For Dirichlet boundaries, we developed two variants:  
	(i) a quad-tree non-uniform stopping time scheme with \(L^\infty\) error of order \(O(\Delta t^{1/2})\), and  
	(ii) a quad-tree uniform stopping time scheme that achieves the improved order \(O(\Delta t)\) through an implicit coupling that preserves an M-matrix structure.  
	For homogeneous Neumann conditions, boundary reflection is modeled via discrete specular reflection, yielding \(L^\infty\) convergence of order \(O(\Delta t^{1/2})\). Under periodic boundaries, modular wrapping together with periodic interpolation restores first-order accuracy, i.e., \(O(\Delta t)\).
	
	Comparisons with the explicit LISL scheme (\( \theta = 0 \)) under Dirichlet conditions highlight the distinct advantages of our framework. While both methods share a probabilistic interpretation, our method naturally incorporates boundary interactions through stopping times and adaptive probabilities, requiring neither extrapolation nor exact
	exterior values . In contrast, the LISL scheme either relies impractically on exact exterior values or may suffer from a loss of accuracy and monotonicity near the boundary when extrapolation is used. The experiments confirm that our approach maintains nearly first-order convergence and robust boundary accuracy without extrapolation, outperforming the LISL method in practical settings.
	
	Stability and positivity are ensured by construction because each update forms a nonnegative linear combination of conditionally expected values and positivity-preserving interpolation operators. All theoretical \(L^\infty\) convergence rates are confirmed numerically for Dirichlet, Neumann and periodic boundaries, demonstrating the method's reliability near domain edges.

	In summary, this research delivers a cohesive numerical methodology that integrates three important features:
	\begin{itemize}
		\item \textbf{Inherent positivity preservation} through a constructive probabilistic discretization.
		\item \textbf{Physics-compatible boundary treatments} for diverse condition types, avoiding accuracy-degrading extrapolation.
		\item \textbf{Rigorous \(L^\infty\) convergence guarantees} under the practical scaling \(\Delta t \sim h\), supported by numerical validation.
	\end{itemize}
	
	The framework therefore offers a robust, stable, and boundary-aware alternative for any anisotropic parabolic problems where classical compact schemes may fail to preserve positivity.
	
	\section*{Acknowledgments}
	
	This work was supported by the Natural Science Foundation of China (no. 12371401).
	
	\bibliographystyle{amsplain}
	
\end{document}